\documentclass[11pt,final,reqno]{amsart}%{article}
\usepackage{amsthm,amsfonts,graphicx,hyperref,tikz,enumerate,psfrag}
\usepackage[a4paper,margin=2.5cm]{geometry}

\usepackage[utf8]{inputenc} 
\RequirePackage{amsthm,amsmath,amsfonts,tikz}

\newtheorem{lemma}{Lemma}

\newtheorem{remark}{Remark}
\newtheorem{proposition}[lemma]{Proposition}
\newtheorem{theorem}[lemma]{Theorem}

\newtheorem{example}{Example}

\newtheorem{corollary}[lemma]{Corollary}

\newcommand{\EE}{{\mathbb{E}}}

\newcommand{\PP}{{\mathbb{P}}}
\newcommand{\bfP}{{\mathbf{P}}}
\newcommand{\Q}{Q}

\newcommand{\D}{\,\rm{d}}
\newcommand{\dE}{\mathbb {E}}
\newcommand{\dP}{\mathbb{P}}

\newcommand{\dN}{\mathbb {N}}
\newcommand{\dR}{\mathbb {R}}

\newcommand{\cF}{\mathcal {F}}

\newcommand{\cT}{\mathcal {T}}

\newcommand{\cI}{\mathcal {I}}

\newcommand{\cN}{\mathcal {N}}
\newcommand{\cO}{\mathcal {O}}

\newcommand{\cE}{\mathcal {E}}
\newcommand{\cG}{\mathcal {G}}

\newcommand{\cL}{\mathcal {L}}
\newcommand{\cP}{\mathcal {P}}

\newcommand{\tx}{{\textsc{tx}}}
\newcommand{\SBRA}[1]{{{\left[#1\right]}}} % [1]
\newcommand{\PAR}[1]{{{\left(#1\right)}}} % (1)
\newcommand{\hh}{%\mathfrak 
{\rm H}}
\newcommand{\ts}{{t}_{\textsc{ent}}}

\newcommand{\veps}{\varepsilon}

\title[Cutoff at the ``entropic time'']{Cutoff at the ``entropic time'' \\for sparse Markov chains}
\author{Charles Bordenave, Pietro Caputo, Justin Salez}

\begin{document}

\begin{abstract}
We study convergence to equilibrium for a class of Markov chains in random environment. The chains are sparse in the sense that in every row of the transition matrix $P$ the mass is essentially concentrated on few entries. Moreover, the entries are exchangeable within each row.  This includes various models of random walks on sparse random directed graphs.  The models are generally non reversible and the equilibrium distribution is itself unknown.
In this general setting we establish the cutoff phenomenon for the total variation distance to equilibrium, with mixing time given by the logarithm of the number of states times the inverse of the average row entropy of $P$. As an application, we consider the case where the rows of $P$ are i.i.d.\ random vectors in the domain of attraction of a Poisson-Dirichlet law with index $\alpha\in(0,1)$. 
 Our main results are based on a detailed analysis of the weight of the trajectory followed by the walker. This approach offers an interpretation of cutoff  as an instance of the concentration of measure phenomenon.  
\end{abstract}

\maketitle
\thispagestyle{empty}

\section{Introduction}
\subsection{Model} 
Let $P$ be a $n\times n$ stochastic matrix with unique invariant law $\pi$. Given an initial  state $i\in[n]=\{1,\ldots,n\}$ and a precision $\varepsilon\in(0,1)$,  the \emph{mixing time} is
\begin{eqnarray*}
t_{\textsc{mix}}^{(i)}(\varepsilon)  \,:= \, \inf\left\{t\in\dN\colon\, \|P^t(i,\cdot)-\pi\|_{\textsc{tv}}\leq \varepsilon\right\},
\end{eqnarray*}
%PC added tot-var distance
where $\|\cdot\|_{\textsc{tv}}$ denotes the total variation distance.
Estimating this quantity is often a difficult task. The purpose of this paper is to relate it to the following simple information-theoretical statistics, which we call the \emph{entropic time}:
\begin{eqnarray}\label{tent}
\ts \, := \, \frac{\log n}{\hh} & \textrm{ where }& \hh \, := \, -\frac{1}{n}\sum_{i,j=1}^n P(i,j)\log {P(i,j)}.
\end{eqnarray}
In words, $\hh$ is the average row entropy of the matrix $P$. Our finding is that, in a certain sense,  ``most'' sparse stochastic  matrices have mixing time roughly given by $\ts$, regardless of the choice of the parameters $\varepsilon\in(0,1)$ and $i\in[n]$. 

To give a precise meaning to the previous assertion, we define the following model of {\em Random Stochastic Matrix}. %PC changed p_{ij} to p_{i,j} everywhere
For each $i\in[n]$, 
let $p_{i,1}\ge \ldots\ge p_{i,n}\ge 0$ be given ranked numbers such that $\sum_{j=1}^np_{i,j}=1$, and define the $n\times n$ random stochastic matrix $P$ by 
\begin{eqnarray}
\label{def:P}
P(i,j)  \,:=\, p_{i,\sigma^{-1}_i(j)}, \qquad \;(1\leq i,j\leq n),
\end{eqnarray}
where $\sigma=(\sigma_i)_{1\leq i\leq n}$ is a collection of $n$ independent, uniform random permutations of $[n]$, which we refer to as the \emph{environment}. We sometimes write $P_\sigma$ instead of $P$ to emphasize the dependence on the environment. Note that the average row entropy $
\hh  =  -\frac 1n\sum_{i,j=1}^np_{i,j}\log p_{i,j}
$ of this random matrix does not depend on the environment. To study large-size asymptotics, we let the input parameters $(p_{i,j})_{1\leq i,j\leq n}$ implicitly depend on $n$ and consider the limit as $n\to\infty$. Our focus is on the sparse and non-degenerate regime defined below. It might help the reader to think of all these parameters as taking values in $\{0\}\cup[\varepsilon,1-\varepsilon]$ for some fixed $\varepsilon\in(0,1)$, so that the number of non-zero entries in each row is bounded independently of $n$. However, we will only impose the following weaker conditions:
\begin{enumerate}%[1.]
\item [$1.$] \textbf{Sparsity} (in every row, the mass is essentially concentrated on a few entries): 
\begin{eqnarray}
\label{assume:sparse}
\hh=\cO(1) & \textrm{ and } & \max_{i\in[n]}\,\sum_{j=1}^n p_{i,j}\left(\log {p_{i,j}}\right)^2 = o(\log n).
\end{eqnarray}
\item  [$2.$] \textbf{Non-degeneracy} (in most rows, the mass is not concentrated on a single entry):  
\begin{eqnarray}
\label{assume:non-degeneracy}
\limsup_{n\to\infty}\left\{\frac{1}{n}\sum_{i,j=1}^n{\bf 1}_{\{p_{i,j}>1-\varepsilon\}}\right\} & \xrightarrow[\varepsilon\to 0^+]{} &  0.
\end{eqnarray}
\end{enumerate}
Note that these conditions imply in particular that $\ts=\Theta(\log n)$ as $n\to\infty$.  

A remark on the asymptotic notation used above and throughout the article: for deterministic sequences of positive numbers $(a_n)$ and $(b_n)$ (with the dependency upon $n$ being often implicit), we write $a_n=o(b_n)$ (resp. $a_n=\cO(b_n)$, $a_n=\Omega(b_n)$, or $a_n= \Theta(b_n)$) to mean that the ratio $a_n/b_n$ vanishes (resp. remains bounded away 
from infinity,  from zero, or from both) as $n\to\infty$. We shall also say that an event that depends on $n$ holds \emph{with high probability} if the probability of this event converges to $1$ as $n\to\infty$; finally, we use  $\xrightarrow[]{\bfP}$ to indicate convergence \emph{in probability}. 

\subsection{Results} Our main result states that around the entropic time $\ts$, the distance to equilibrium undergoes the following sharp transition, henceforth referred to as a \emph{uniform cutoff} (to emphasize the insensitivity to the initial state). 
\begin{theorem}[Uniform cutoff at the entropic time] \label{th:main}
Under the above assumptions, the Markov chain defined by $P$ has, with high probability, a unique stationary distribution $\pi$. Moreover, for any fixed $\varepsilon\in(0,1)$, we have 
\begin{eqnarray*}
\max_{i\in[n]}\,\left|\frac{t_{\textsc{mix}}^{(i)}(\varepsilon)}{\ts}-1\right| & \xrightarrow[n\to\infty]{\bfP} & 0.
\end{eqnarray*}
In other words,  for $t=\lambda \ts+o(\ts)$ with  $\lambda$ fixed as $n\to\infty$, we have the following transition:
\begin{eqnarray}
\lambda<1 \qquad \Longrightarrow \qquad \min_{i\in[n]} \|P^t(i,\cdot)-\pi\|_{\textsc{tv}} & \xrightarrow[n\to\infty]{\bfP} & 1
\label{cutoff1}\\
\lambda>1 \qquad \Longrightarrow \qquad \max_{i\in[n]} \|P^t(i,\cdot)-\pi\|_{\textsc{tv}} & \xrightarrow[n\to\infty]{\bfP} & 0.
\label{cutoff2}\end{eqnarray}
\end{theorem}

Let us first illustrate our result with a special case. 

\begin{example}[Random walk on random digraphs]\label{ex:dout} When $p_{i,1}=\ldots=p_{i,d_i}=\frac{1}{d_i}$ and $p_{i,d_i+1}=\cdots=p_{i,n}=0$ for some integers $d_1,\ldots,d_n\ge 1$,  the random matrix $P$ may be interpreted as the transition matrix of the random walk on a uniform random directed graph with $n$ vertices and out-degrees $d_1,\ldots,d_n$ (loops are allowed). The average row entropy is then simply the average logarithmic out-degree $\hh=\frac 1n\sum_{i=1}^n\log d_i$. Assumption (\ref{assume:sparse}) requires that this average remain bounded as $n\to\infty$, and also that the maximum out-degree $\Delta$ satisfies $\Delta =e^{o(\sqrt{\log n})}$. Assumption (\ref{assume:non-degeneracy}) simply asks for the proportion of
out-degree-one vertices to vanish. Notice that, because of the possibility of vertices with zero in-degree,  the random matrix $P$ may, with uniformly positive probability, fail to be irreducible.  However, under the above conditions,  Theorem \ref{th:main} ensures that with high probability there is a unique stationary distribution and the  walk exhibits uniform cutoff at time $(\log n)/\hh$. To the best of our knowledge, the occurence of a cutoff phenomenon is new even in the  special case where $d_1=\cdots=d_n=r$ for some fixed integer $r\ge 2$, known as the \emph{random $r-$out digraph}. We emphasize that the results of \cite{bordenave2015random} do not apply here, since with high probability  the minimum in-degree is zero and the maximum in-degree diverges (logarithmically)  
with $n$. We also note that the structure of the stationary measure $\pi$ on random $r-$out digraphs has been investigated in details by Addario-Berry, Balle and Perarnau  \cite{2015arXivRout}. 
\end{example}

Interesting illustrations of Theorem \ref{th:main} can be obtained by taking the input parameters $(p_{i,j})$ also random. In fact, any random stochastic matrix whose law is invariant under permutation of entries within each row is a mixture of random matrices of the form (\ref{def:P}) and is therefore eligible for an application of Theorem \ref{th:main} (conditionnally on the $(p_{i,j})$), provided the assumptions \eqref{assume:sparse}-\eqref{assume:non-degeneracy} are satisfied with high probability. The following theorem illustrates this with the case where the rows $\{(p_{i,1},\dots,p_{i,n}), \,i=1,\dots, n\}$ are i.i.d.\ random vectors in the domain of attraction of a Poisson-Dirichlet law. The spectral properties of this natural random stochastic matrix were investigated in \cite{2016arXiv161001836B}, see Section \ref{sec:related} below for more details.
% (see Section (\ref{sec:heavytails}) for more details). 
\begin{theorem}[Random walk in a heavy-tailed environment]
\label{th:heavytails}
Let $\omega=\left(\omega_{ij}\right)_{1\leq i,j< \infty}$ be i.i.d.\ positive random variables whose tail distribution function $G(t)=\PP(\omega_{ij}>t)$ is regularly varying at infinity with index $\alpha\in (0,1)$, i.e., for each $\lambda>0$,
\begin{eqnarray}
\label{regvar}
\frac{G(\lambda t)}{G(t)} & \xrightarrow[t\to\infty]{} & \lambda^{-\alpha}.
\end{eqnarray} 
Then as $n\to\infty$, the $n-$state Markov chain with transition matrix 
\begin{eqnarray}\label{pmat}
{P}(i,j)  :=  \frac{\omega_{ij}}{\omega_{i1}+\cdots+\omega_{in}},\;\qquad (1\leq i,j\leq n)
\end{eqnarray} has with high probability a unique stationary distribution $\pi$,  and exhibits uniform cutoff at time $\frac{\log n}{h(\alpha)}$ in the sense of \eqref{cutoff1}-\eqref{cutoff2}, where $h(\alpha)$ is defined in terms of the {digamma} function $\psi=\frac{\Gamma'}{\Gamma}$  by
\begin{eqnarray}\label{ha}
h(\alpha)  :=  \psi(1)-\psi(1-\alpha) \ = \ \int_0^{\infty}\frac{e^{\alpha t}-1}{e^t-1}\, \D t.
\end{eqnarray}
\end{theorem}

Let us now briefly sketch  the main ideas behind the proof of our results.  
\subsection{Proof outline}
\label{sub:proxy}
The essence of the sharp transition described in Theorem \ref{th:main} lies in a quenched \emph{concentration of measure} phenomenon in the trajectory space that  can be roughly described as follows; we refer to Section \ref{sec:concentration} for more details. Let $i=X_0,X_1,X_2,\dots$ denote the trajectory of the random walk with transition matrix $P$ and starting point $i\in[n]$ and let $Q_i$ denote the associated quenched law, that is the law of the trajectory for a fixed realization of the environment $\sigma$. Define the trajectory weight
 \begin{eqnarray*}
\rho(t):= P(X_0,X_1)\cdots P(X_{t-1},X_t).
\end{eqnarray*} 
In  other words, $\rho(t)$ is the probability of the followed trajectory   up to time $t$. 
Theorem \ref{pr:weight} below establishes that for $t=\Theta(\log n)$, with high probability with respect to the environment, it is very likely, uniformly in the starting point $i$, that $ \log \rho(t)\sim -\hh t $. More precisely, we prove that for any $\varepsilon>0$, 
 \begin{eqnarray}\label{concentra}
\max_{i\in[n]}\,Q_i\left(\rho(t)\notin\left[e^{-(1+\varepsilon)\hh t},e^{-(1-\varepsilon)\hh t}\right]\right) \;\xrightarrow[n\to\infty]{\bfP} \; 0.
\end{eqnarray} 
 In particular, at $t=\ts$ one has $\log \rho(t)\sim -\log n$. As we will see in Section \ref{sec:lower}, the lower bound \eqref{cutoff1} is a rather direct consequence of the concentration result \eqref{concentra}. Indeed, we will check that if the invariant probability measure has its atoms $\pi(j), j \in [n],$ roughly of order $\cO(1/n)$ then we cannot have reached equilibrium by time $t$ if with high probability $\rho(t)\gg 1/n$. The proof of the upper bound \eqref{cutoff2} requires a more detailed investigation of the 
% tree-like behavior 
structure of the set of  trajectories that the random walker is likely to follow. As explained in Section \ref{sec:upper}, 
this allows us to obtain a sharp comparison between  the transition probability $P^t(i,j)$ and a certain approximation of $\pi(j)$. Note that the true stationary distribution $\pi$ appearing in Theorem \ref{th:main} is a non-trivial 
%$n$-dependent 
random object, with no explicit expression. To overcome this difficulty, we will actually prove \eqref{cutoff1}-\eqref{cutoff2} with $\pi$ replaced by the more tractable approximation 
\begin{eqnarray}
\label{proxy}
\widehat{\pi}(j) \, := \, \frac{1}{n}\sum_{i\in[n]}P^{\lfloor\frac{\ts}{10}\rfloor}(i,j).
\end{eqnarray}
The choice $\frac{\ts}{10}$ is not particularly important:  any probability distribution $\widehat{\pi}$ for which we manage to prove  \eqref{cutoff1}-\eqref{cutoff2} suffices to guarantee the original claim (see forthcoming Remark \ref{rq:constant}). Indeed, using the stationarity of $\pi$ and the convexity of $\|\cdot\|_{\textsc{tv}}$, for all $t\in\dN$ we may write %PC added for all $t\in\dN$ and _{TV}
\begin{eqnarray*}
\|\pi-\widehat{\pi}\|_{\textsc{tv}} & = & \|\pi P^t-\widehat{\pi}\|_{\textsc{tv}}\\
& \le & \sum_{i\in[n]}\pi(i)\|P^t(i,\cdot)-\widehat{\pi}\|_{\textsc{tv}}\\
& \le & \max_{i\in[n]}\|P^t(i,\cdot)-\widehat{\pi}\|_{\textsc{tv}}.
\end{eqnarray*}
Consequently, if \eqref{cutoff1}-\eqref{cutoff2} hold for $\widehat{\pi}$, then we automatically obtain \begin{eqnarray}
\label{unique}
\|\pi-\widehat{\pi}\|_{\textsc{tv}} &  \xrightarrow[n\to\infty]{\bfP} & 0,
\end{eqnarray}
and we may therefore safely replace $\widehat{\pi}$ by $\pi$ in \eqref{cutoff1}-\eqref{cutoff2} to recover the original claim. Also, the fact that \eqref{unique}  applies simultaneously to all stationary distributions forces the latter to be unique with high probability. Indeed, it is classical that if a Markov chain admits at least two stationary probability distributions, then one can always choose them to be supported on distinct communication classes, so that their total-variation distance is $1$.

 \subsection{Related work}
 \label{sec:related}
 \begin{figure}[h!]
\begin{center}
\includegraphics[angle =0,height = 10cm]{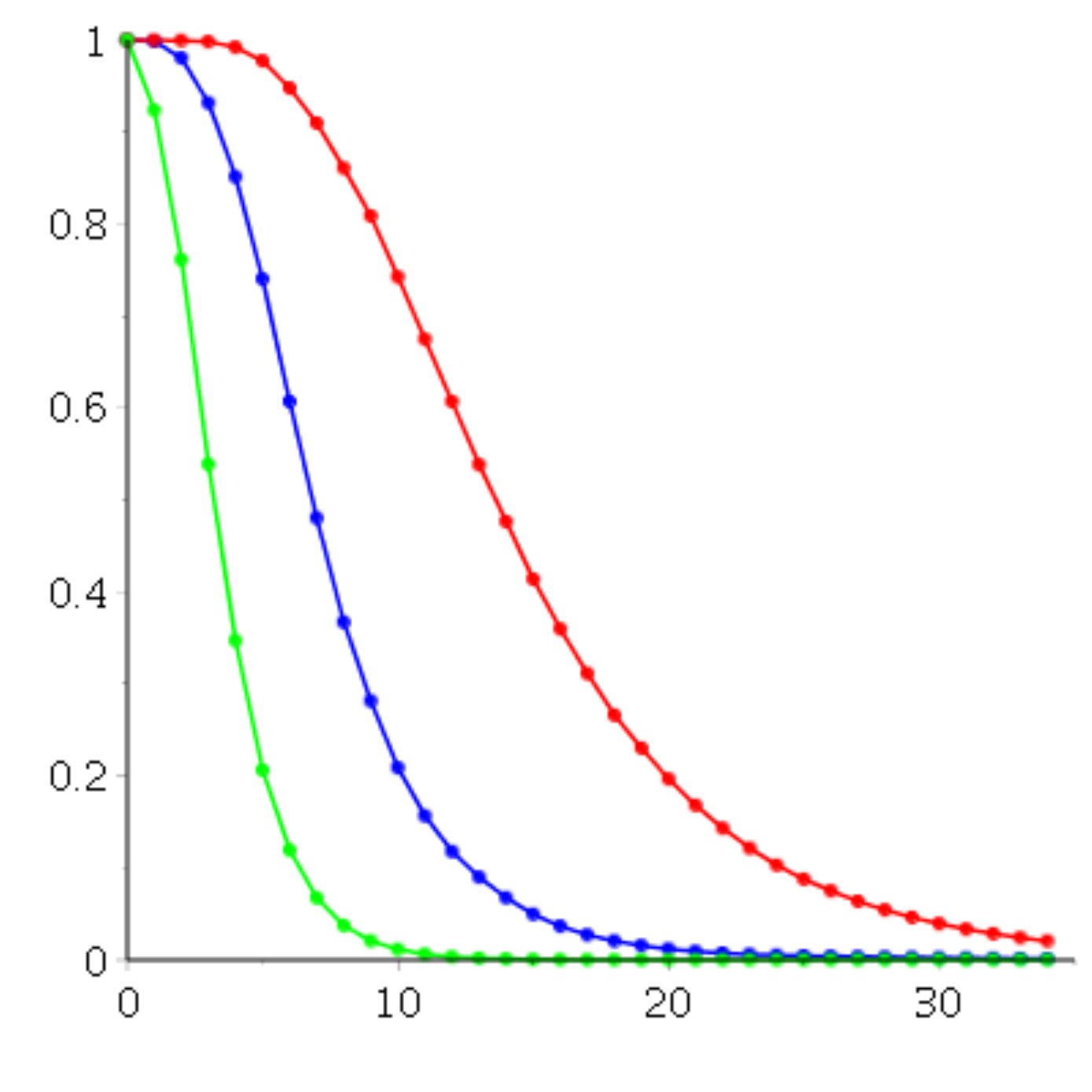}
\caption{Distance to equilibrium along time for the $n\times n$ random matrix in Theorem \ref{th:heavytails}, with Pareto($\alpha$) entry distribution, i.e. $\PP(w_{ij}>t)=(t\vee 1)^{-\alpha}$. Here, $n=10^4$ and $\alpha=0.3$ (red), $\alpha=0.5$ (blue) and $\alpha=0.7$ (green). Note that the function $h$ increases continuously from %PC changed h(\infty)=\infty --> h(1)=\infty etc
$h(0)=0$ to $h(1)=\infty$: the more ``spread-out'' the transition probabilities are, the faster the chain mixes. }
\label{fig:distance}
\end{center}
\end{figure}
Theorem \ref{th:main} describes a sharp transition in the approach to equilibrium, visible on Figure \ref{fig:distance}: the total variation distance drops from the maximal value $1$ to the minimal value $0$ on a time scale that is asymptotically negligible with respect to the mixing time. This is an instance of the so-called cutoff phenomenon, a remarkable property shared by several models of finite Markov chains. Since its original discovery by Diaconis, Shashahani, and  Aldous in the context of card shuffling around 30 years ago \cite{diaconis1981generating,aldous1983mixing,aldous1986shuffling}, the problem of characterizing the Markov chains exhibiting cutoff has attracted much attention. We refer to \cite{diaconis1996cutoff,aldous2002reversible,levin2009markov} for an introduction. 
While the phenomenon is now rather well understood in various specific settings, see e.g.\ \cite{DS,ding2010total} for the case of birth and death chains, a general characterization is still unknown and its nature remains somewhat elusive (but see \cite{2014arXiv1409.3250B} for an interesting interpretation in the reversible case). %Up to recently, known examples of Markov chains exhibiting cutoff were mostly restricted to reversible models such as birth and death processes and random walks on groups; 

Recently, part of the attention has shifted from ``specific'' to ``generic'' instances: instead of being fixed, the sequence of transition matrices itself is drawn at random from a certain distribution, and the cutoff phenomenon is shown to occur for almost every realization. Examples include certain random birth and death chains \cite{MR3068032,RSA20693}, ``random random walks'' on some finite groups \cite{MR2121795,MR1465637}, or the simple/non-backtracking random walk on various models of sparse random graphs, including random regular graphs \cite{lubetzky2010cutoff}, graphs with given degrees \cite{2015arXiv,2015arXivNBRW}, and the giant component of the Erd\"os-Renyi random graph \cite{2015arXiv}. The above mentioned references are all concerned with the reversible case of undirected graphs, where the associated simple random walk and non-backtracking random walk have explicitly known stationary distributions. In our recent work \cite{bordenave2015random}, we investigated the non-reversible case of random walk on sparse directed graphs with given bounded degree sequences. Despite the lack of direct information on the stationary distribution, we obtained a detailed description of the cutoff behavior in such cases. 

The present paper considerably extends the results in \cite{bordenave2015random} by establishing cutoff for a large class of non-reversible sparse stochastic matrices, not necessarily arising as the transition matrix of the random walk on a (directed) graph. The time-irreversibility actually plays a crucial role in our proofs: despite the lack of an explicit underlying structure, the Markov chains that we consider turn out to  exhibit a spontaneous ``non-backtracking" tendency which allows us to establish a certain i.i.d. approximation for the environment seen by a typical walker. While the overall strategy of proof of our main result is closely related to the one we introduced in \cite{bordenave2015random}, the level of generality allowed for in the transition probabilities requires an entirely new analysis of path weights. 
For instance, one of the features making the control of trajectory weights much more challenging here is the lack of nontrivial upper bounds on the probability of any particular transition (as opposed to the model studied in \cite{bordenave2015random} where the minimal out-degree was assumed to be at least $2$). 

Our entropic time $\ts$ admits a natural interpretation as follows. One could easily deduce from our proofs that, with high probability, the entropy of the distribution $P^{t} (i, \cdot)$ on the time-interval $[0,\ts]$ grows roughly linearly, at rate $\hh$. This in turns implies that the entropy of ${\pi}$ is $(1-o(1) ) \log n$ with high probability. Consequently, we see that the cutoff occurs precisely when the entropy of the chain reaches the entropy of the invariant distribution, and that the mixing time is given by the entropy at stationarity divided by the average single step entropy $\hh$. Interestingly, the same interpretation can be given to the main results in the models studied in \cite{lubetzky2010cutoff,2015arXiv,2015arXivNBRW,bordenave2015random}. It is thus perhaps tempting to believe that this scenario should apply to a much larger class of Markov chains in random environments, although we do not have a  precise conjecture to propose at the present time.

As already mentioned, the stationary measure $\pi$ appearing in our general theorem is a delicate random object with no explicit expression. In the special case of Example \ref{ex:dout} where all out-degrees are equal however, the structure of $\pi$ has been investigated in details by Addario-Berry, Balle and Perarnau  \cite{2015arXivRout}. 
Concerning the heavy tailed model of Theorem \ref{th:heavytails}, we point out that the eigenvalues and singular values of the random stochastic matrix \eqref{pmat} were analyzed recently in % by Bordenave, Caputo, Chafaï and Piras 
 \cite{2016arXiv161001836B}: under a slightly stronger assumption than (\ref{regvar}), the associated empirical distributions are shown to converge to some deterministic limit, depending only on $\alpha$, and characterized by a certain recursive distributional equation. The numerical simulations given therein seem to indicate that the spectral gap should also converge to a non-zero limit, and the authors formulate an explicit conjecture 
 %is even formulated 
 (see \cite[Remark 1.3]{2016arXiv161001836B}).  However, the results in  \cite{2016arXiv161001836B} do not allow one to infer something quantitative about the distance to equilibrium. Indeed, the relation between spectrum and mixing for non-reversible chains is rather loose, and one would certainly need more precise information on the structure of the eigenvectors -- as done in, e.g., \cite{2015arXiv150704725L}. The proof of Theorem \ref{th:heavytails} relies entirely on the 
the general result of Theorem \ref{th:main} and makes no use of spectral theory. As detailed in Lemma \ref{lm:beta} below,  the expression \eqref{ha} for $h(\alpha)$ coincides with the expected value of $-\log\xi$ where $\xi$ has law Beta$(1-\alpha,\alpha)$. That should be expected in light of the fact that a size-biased pick from the Poisson Dirichlet law is Beta-distributed \cite{pitman1997two}.     

\section{Quenched law of large numbers for path weights}
\label{sec:concentration}
The main result of this section can be interpreted as a quenched law of large numbers for the logarithm of the total weight of the path followed by the random walk; see   Theorem \ref{pr:weight} below.   

\subsection{Uniform unlikeliness}

Consider a collection $\sigma=(\sigma_i)_{1\leq i\leq n}$ of $n$ independent random permutations, referred to as  the environment, 
and a $[n]-$valued process $X=(X_t)_{t\ge 0}$ whose conditional law, given the environment, is that of a Markov chain with transition matrix (\ref{def:P}) and initial law uniform on $[n]$.  Our main object of interest will be the sequence of \emph{weights} $W=(W_t)_{t\ge 1}$  seen along the trajectory, and the associated total weight up to time $t$:
\begin{align}
\label{weight}
W_t  := & P(X_{t-1},X_{t})\,,\quad \quad\rho(t) : =\prod_{s=1}^t W_s.
\end{align}
Write $Q$ for the conditional law of the pair $(X,W)$ given the environment. Note that it is a random probability measure on the \emph{trajectory space} $\cE=[n]^{\{0,1,\dots\}}\times[0,1]^{\{1,2,\dots\}}$ equipped with the natural product $\sigma$-algebra of events. A generic point of $\cE$ will be denoted $(x,w)$, where $x=(x_0,x_1,x_2,\ldots)$ and $w=(w_1,w_2,\ldots)$. For example, the trajectorial event ``a transition with weight less than $n^{-\gamma}$ occurs within the first $t$ steps'' will be denoted 
\begin{align}\label{eventa}%PC changed A --> A_0
& A_0  =  \left\{(x,w)\in \cE:\, \min(w_1,\ldots, w_t) < n^{-\gamma} \right\}, \\
& \qquad \textrm{and}  \quad Q(A_0)  =  \frac{1}{n}\sum_{i_0\in[n]}\cdots\sum_{i_t\in[n]}\,\prod_{s=1}^t P(i_{s-1},i_s)\left(1-\prod_{u=1}^t{\bf 1}_{\left\{P(i_{u-1},i_u)\geq n^{-\gamma}\right\}}\right). \nonumber
\end{align}
We let also $Q_i(\cdot):=Q(\cdot|X_0=i)$ be the law starting at $i\in[n]$. Recall that all objects are implicitly indexed by the size-parameter $n$, %: we actually have one instance for each $n\ge 1$, 
and asymptotic statements are understood in the $n\to\infty$ limit. We call a trajectorial event $A$ \emph{uniformly unlikely} if
\begin{eqnarray}
\label{typical}
\max_{i\in[n]}\,Q_i\left(A\right)  \xrightarrow[n\to\infty]{\bfP}  0.
\end{eqnarray}
\begin{lemma}\label{lm:low} For $t=\cO(\log n)$ and $\gamma=\Theta(1)$, the event $A=A_0$ from \eqref{eventa} 
%PC added A = A_0
%$\left\{\min(w_1,\ldots,w_t) < n^{-\gamma}\right\}$ 
is uniformly unlikely.
 \end{lemma}
\begin{proof} 
A union bound implies the deterministic estimate %we clearly have the deterministic bound
\begin{eqnarray*}
\max_{i\in[n]}\,Q_i(A_0)  \leq  t\, \max_{i\in[n]}\,\sum_{j=1}^np_{i,j}{\bf 1}_{\{p_{i,j}< n^{-\gamma}\}}.
\end{eqnarray*}
Since $u\mapsto (\log {u})^2$ is decreasing on $(0,1)$,
\begin{eqnarray*}
\max_{i\in[n]}\,Q_i(A_0) \leq \frac{t}{(\gamma\log n)^2}\, \max_{i\in[n]}\,\sum_{j=1}^n\, p_{i,j}\left(\log {p_{i,j}}\right)^2.
\end{eqnarray*}
 The conclusion follows from the assumption (\ref{assume:sparse}). 
\end{proof}

 Our main task in the rest of this section 
will be to establish:

\begin{theorem}[Trajectories of length $t$ have weight roughly $e^{-\hh t}$]\label{pr:weight}For $t=\Theta(\log n)$ and fixed $\varepsilon>0$, the event $\left\{\rho(t) \notin\left[e^{-(1+\varepsilon)\,\hh\, t},e^{-(1-\varepsilon)\,\hh \,t} \right]\right\}$ is uniformly unlikely.
\end{theorem}
Let us observe here, for future reference, that if $\theta\colon \cE\to \cE$ is the operator that shifts  $x=(x_0,x_1,\ldots)$ and $w=(w_1,w_2,\ldots)$ to $x'=(x_1,x_2,\ldots)$ and $w'=(w_2,w_3,\ldots)$ respectively, then, for any $i\in[n]$, $t\in\dN$ and any event $A\subset \cE$ 
\begin{eqnarray}
\label{rk:timeshift}
Q_i(\theta^{-t}(A)) & = & \sum_{j\in[n]}P^t(i,j)Q_j(A) \ \leq \ \max_{j\in[n]}\,Q_j(A),
\end{eqnarray}
where $\theta^{-t}(A)=\{(x,w)\in \cE: \, \theta^t(x,w)\in A\}$. Thus, if $A$ is uniformly unlikely, then so is $\theta^{-t}(A)$, for any choice of $t=t(n)$.

\subsection{Sequential generation}
\label{sec:seq}
%Given an event $\cA$ in the trajectory space, b
By averaging the quenched probability $Q(\cdot)$ with respect to the environment, one obtains the so-called \emph{annealed} probability, which we denote by $\PP$. In symbols, letting $\EE$ denote the associated expectation, for any event $A$ in the trajectory space:
\begin{eqnarray*}
%\label{first-moment}
\EE[Q(A)]= \frac1n\sum_{i=1}^n\EE[Q_i(A)] = \PP\left((X,W)\in A\right).
\end{eqnarray*}  
%We will use the symbols $\PP$ and $\EE$ for the probability and expectation over the environment. 
Markov's inequality offers a way to reduce the problem of estimating the worst-case quenched probability  $\max_{i\in[n]}Q_{i}(A)$ of a trajectorial event $A\subset \cE$ to that of controlling the corresponding annealed quantity, at the cost of an extra factor of $n$: for any $\delta>0$,
\begin{eqnarray}
\label{first-moment}
\PP\left(\max_{i\in[n]}\Q_i(A)>\delta\right)  \leq  \frac{1}{\delta}\,\EE\left[\sum_{i=1}^nQ_i(A)\right]= \frac{n}{\delta}\,\PP\left((X,W)\in A\right).
\end{eqnarray}
The analysis of the right-hand side may often be simplified by the observation that the pair $(X,W)$ can be constructed sequentially, together with the underlying environment $\sigma$, as follows: initially, $\textrm{Dom}(\sigma_i)=\textrm{Ran}(\sigma_i)=\emptyset$ for all $i\in[n]$, and $X_0$ is drawn uniformly from $[n]$; then for each $t\ge 1$, 
\begin{enumerate}
\item[$\#1.$] Set $i=X_{t-1}$ and draw an index $j\in [n]$ at random with probability $p_{i,j}$.
\item[$\#2.$] If $j\notin\textrm{Dom}(\sigma_i)$, then extend $\sigma_i$ by setting  $\sigma_i(j)=k$, where $k$ is uniform in $[n]\setminus \textrm{Ran}(\sigma_i)$.
\item[$\#3.$] In either case, $\sigma_i(j)$ is now well defined: set $X_t=\sigma_i(j)$ and $W_t=p_{i,j}$.
\end{enumerate}
Let us illustrate the strength of this sequential construction  on an important trajectorial feature. A path $(x_0,\ldots,x_t)\in[n]^{t+1}$ naturally induces a directed graph with vertex set $V=\{x_0,\ldots,x_t\}\subset[n]$ and edge set $E=\{(x_{0},x_1),\ldots,(x_{t-1},x_t)\}\subset [n]\times[n]$. As a rule, below we neglect possible multiplicities in the edge set $E$, that is every repeated edge from the path appears only once in $E$.  We define the  \emph{tree-excess} of the path $(x_{0},\ldots,x_t)$ as  
\begin{eqnarray*}
\tx(x_0,\ldots,x_t) & = &  1 + |E|-|V|.
\end{eqnarray*}
Here $|V|$ and $|E|$ denote the cardinalities of $V$ and $E$. 
In particular, $\tx(x_0,\ldots,x_t)=0$ if and only if $(x_0,\ldots,x_t)$ is a \emph{simple path} in the usual graph-theoretical sense, while $\tx(x_0,\ldots,x_t)=1$ if and only if the edge set of  $(x_0,\ldots,x_t)$ has a single cycle (the path may turn around it more than once). 
\begin{lemma}[Tree-excess]\label{lm:tx} For $t=o(n^{1/4})$, 
$\left\{\tx(X_0,\ldots,X_t) \ge 2\right\}$ is uniformly unlikely.
\end{lemma}
\begin{proof}
In the sequential generation process, we have 
$
\tx(X_0,\ldots,X_t)  =  \xi_1+\dots+\xi_t,
$
where $\xi_t\in\{0,1\}$ indicates whether or not, during the $t^{\textrm{th}}$ iteration, the random index $k$ appearing at line $\#2$ is actually drawn and satisfies $k\in\{X_0,\ldots,X_{t-1}\}$. The conditional chance of this, given the past, is at most  
\begin{eqnarray*}
\frac{\left|\{X_0,\ldots,X_{t-1}\}\right|-\left|\textrm{Ran}(\sigma_i)\right|}{n-\left|\textrm{Ran}(\sigma_i)\right|} &  \leq  & \frac{t}{n}.
\end{eqnarray*}
Thus, $\tx(X_0,\ldots,X_t)$ is stochastically dominated by a Binomial $(t,\frac{t}{n})$. In particular,  
\begin{eqnarray}
\label{bound:tx}
\PP\left(\tx(X_0,\ldots,X_t)\ge 1\right)  \, \leq\, \frac{t^2}n \qquad\textrm{ and }\qquad \PP\left(\tx(X_0,\ldots,X_t)\ge 2\right)  \, \leq\, \frac{t^4}{n^2}.
\end{eqnarray}
Now, let $n\to\infty$: since $t=o(n^{1/4})$, we have $\frac{t^4}{n^2}=o(\frac 1n)$  and (\ref{first-moment}) concludes the proof.
\end{proof}

\subsection{Approximation by i.i.d.\ samples}

 Consider the modified process $(X^\star,W^\star)$  obtained by resetting $\textrm{Dom}(\sigma_i)=\textrm{Ran}(\sigma_i)=\emptyset$ before every execution of line $\#2$, thereby suppressing any time dependency: the environment is locally regenerated \emph{afresh} at each step. In particular, the pairs $(X^\star_{t-1},W^\star_t)_{t\ge 1}$ are i.i.d.\ with law 
\begin{eqnarray}
\label{size-biased}
\PP\left(X^\star_0=i,W^\star_1 \geq w  \right)  =   \sum_{j =1}^n \frac{p_{i,j}}{n} { \bf 1}_{\{p_{i,j}\geq  w\}} \qquad (i \in [n], w \geq 0).
\end{eqnarray} 
By construction, the modified process and the original one can be coupled in such a way that they coincide until the  time \begin{eqnarray}
\label{Tcp}T := \inf\{t\ge 0\colon \tx(X_0,\ldots,X_{t})=1\},
\end{eqnarray} that is the first time  a state gets visited for the second time. Thus, on the event $\{T\ge t\}$, 
\begin{eqnarray}
\label{coupling}
(X_0^\star,\ldots,X_t^\star)=(X_0,\ldots,X_t) & \textrm{ and } & (W_1^\star,\ldots,W_t^{\star})=(W_1,\ldots,W_t).
\end{eqnarray}
We exploit this observation to establish a preliminary step towards Theorem \ref{pr:weight}. Notice that the estimate below becomes trivial if the parameters $p_{i,j}$ are such that $p_{i,j}\leq 1-\varepsilon$ for some fixed $\varepsilon>0$. In the general case it relies on the non-degeneracy assumption \eqref{assume:non-degeneracy}.  
\begin{lemma}
\label{lm:high}
If $t=\Theta(\log n)$ and $\delta=o(1)$, then $\left\{\rho(t)> e^{-\delta t}\right\}$ is uniformly unlikely. 
\end{lemma}
\begin{proof}
Call $(x_0,\ldots,x_t)$ a \emph{cycle} if $(x_0,\ldots,x_{t-1})$ is simple and $x_t=x_0$. We will show:
\begin{enumerate}[(i)]
\item for $t$ and $\delta$ as above, $B:=\{\rho(t)> e^{-\delta t}, \tx(X_0,\ldots,X_t)=0\}$ is uniformly unlikely;
\item for $\delta=o(1)$, $C_\delta:=\left\{\exists s\ge 1\colon (X_0,\ldots,X_s)\textrm{ is a cycle}, \rho(s)> e^{-\delta s}\right\}$ is uniformly unlikely.
\end{enumerate}
Let us first show that this is sufficient to conclude the proof. Indeed, the event 
$$A:=\{\rho(t)>e^{-\delta t}\}=\{(x,w)\in \cE:\, w_1\cdots w_t>e^{-\delta t}\},$$
can be partitioned according to the size of $\tx(x_0,\dots,x_t)$. Therefore 
$$
A\subset B \cup \{\tx(x_0,\dots,x_t)=1,\,w_1\cdots w_t>e^{-\delta t}\}\cup  \{\tx(x_0,\dots,x_t)\geq 2\}
\,.
$$
The event $\{\tx(x_0,\dots,x_t)\geq 2\}$ is uniformly unlikely, thanks to Lemma \ref{lm:tx}.
The event $B$ is uniformly unlikely, by (i) above. The event $\{\tx(x_0,\dots,x_t)=1,\,w_1\cdots w_t>e^{-\delta t}\}$ on the other hand is contained in the union of the following three events:
\begin{itemize}
\item $\left\{\tx\left(x_0,\ldots,x_{\lfloor t/3\rfloor }\right)=0, \,w_1\cdots w_{\lfloor t/3\rfloor}>e^{-\delta t}\right\}$
\item $\left\{\tx\left(x_{\lceil 2t/3\rceil},\ldots,x_t\right)=0, \,w_{\lceil 2t/3\rceil}\cdots w_t>e^{-\delta t}\right\}$ 
 \item $\left\{\tx\left(x_0,\ldots,x_t\right)=\tx(x_0,\ldots,x_{\lfloor t/3\rfloor})=\tx(x_{\lceil 2t/3\rceil},\ldots,x_{t})=1, \,w_{1}\cdots w_t>e^{-\delta t}\right\}$
\end{itemize}
The first two cases are uniformly unlikely by (i) and by the observation (\ref{rk:timeshift}).
To handle the third case, observe that 
if $\tx(x_0,\dots,x_t)=1$, then
the path $(x_0,\dots,x_t)$ can be rewritten as $(x_0,\dots,x_a,\dots,x_{a+r\ell},\dots,x_t)$, where $(x_0,\dots,x_a)$ and $(x_{a+r\ell},\dots,x_t)$ are simple paths, while the path  $
( x_a,\dots,x_{a+r\ell})$ consists of $r$ complete turns around a cycle of length $\ell$. Here $a\geq 0$, $r,\ell\geq 1$ and $a+r \ell\leq t$.  If $\tx(x_0,\ldots,x_{\lfloor t/3\rfloor})=\tx(x_{\lceil 2t/3\rceil},\ldots,x_{t})=1$, then the two simple paths must have lengths less than $t/3$ and therefore $r \ell >  t/3$. 
If $\rho=w_{a+1}\cdots w_{a+\ell}$ is the weight associated to one turn around the cycle, then  $w_1\cdots w_t>e^{-\delta t}$ implies $\rho^r>e^{-\delta t}$ and therefore $\rho> e^{-3\delta\ell}$. It follows that the shifted trajectory $\theta^{\lfloor t/3\rfloor}(x,w)$ must belong to $C_{3\delta}$. 
Using  (\ref{rk:timeshift}) and (ii) above, this is uniformly unlikely. 

It remains to prove (i) and (ii). 
By (\ref{coupling}), we have $$\textstyle{\left\{(X,W)\in B\right\}\subset\left\{W_1^\star\cdots W_t^\star > e^{-\delta t}\right\}\subset\left\{\sum_{s=1}^t{\bf 1}_{\{W_s^\star<e^{-2\delta}\}} < \frac{t}{2}\right\}}.$$ Now $\sum_{s=1}^t{\bf 1}_{\{W_s^\star<e^{-2\delta}\}}$ is  Binomial$(t,q)$ with 
$q = \PP\left(W_1^\star<e^{-2\delta}\right)$. 
  Thus, Bennett's inequality yields  
\begin{eqnarray}
\label{benett1}
\PP\left((X,W)\in B\right)  \le  e^{-t\phi\left(q\right)},
\end{eqnarray}
for some universal function $\phi\colon[0,1]\to\dR_+$ that diverges at $1^-$ (more precisely,  \cite[Theorem 2.9]{MR3185193} gives $\phi(q) = \sigma^2 h ( (q-1/2) / \sigma^2 )$ with $\sigma^2 =q (1 - q) $, $h(x) = (x+1) \log (x+1) - x$ for $x \geq 0$ and $0$ otherwise). From \eqref{size-biased}, 
$$1-q = \PP\left(W_1^\star\geq e^{-2\delta}\right) = \frac1n\sum_{i,j=1}^n
p_{i,j}{\bf 1}_{\{p_{i,j}\geq e^{-2\delta}\}}.$$
Now, let $n\to\infty$. Since $\delta\to 0$, the assumption (\ref{assume:non-degeneracy}) ensures that $q\to 1$, so that 
\eqref{benett1} implies $\PP\left((X,W)\in B\right)=o(\frac1n)$. From the first moment argument \eqref{first-moment} one obtains part (i). 

To prove part (ii), observe that the coupling \eqref{coupling} implies 
$$\{(X,W)\in C_\delta\}\subset \bigcup_{s\ge 1}\{W_1^\star\cdots W_s^\star > e^{-\delta s}, X_s^\star=X_0^\star\}.$$ Since $X_s^\star$ is independent of the other variables and uniform, the  argument for \eqref{benett1} shows that
\begin{eqnarray}
\label{benett2}
\PP\left((X,W)\in C_\delta\right)  \leq  \frac{1}{n}\sum_{s\ge 1}e^{-s\,\phi(q)}  =  \frac{1}{n(e^{\phi(q)}-1)}.
\end{eqnarray}
Letting $n\to\infty$, the conclusion follows as above. 
%Now, let $n\to\infty$. In view of Assumption (\ref{assume:non-degeneracy}), the fact that $\delta\to 0$ ensures that $q\to 1$, and hence that $\phi(q)\to\infty$. Thus, the right-hand sides of (\ref{benett1})-(\ref{benett2}) are $o(1/n)$. This establishes (i) and (ii). 
\end{proof}
\subsection{Proof of Theorem \ref{pr:weight}}
The event $A=\left\{\rho(t) \notin\left[e^{-(1+\varepsilon)\,\hh\, t},e^{-(1-\varepsilon)\,\hh \,t} \right]\right\}$ can be written as 
\begin{eqnarray*}
A=\left\{\left|1-\frac{1}{\hh\, t}\sum_{s=1}^t\log\frac{1}{W_s}\right| > \varepsilon\right\}.
\end{eqnarray*} 
We are going to prove 
the uniform unlikeliness of $A$ for any fixed $\varepsilon>0$ and $t=\Theta(\log n)$. 
First note that, by (\ref{bound:tx}), the random time $T$ defined in \eqref{Tcp} satisfies  $$\PP(T\leq t)\leq \frac{t^2}n=o(1).$$  Combining this with (\ref{coupling}), we see that 
\begin{eqnarray}\label{weas}
\PP\left((X,W)\in A\right) %& = & \PP\left(\left|1-\frac{1}{\hh t}\sum_{s=1}^t\log\frac{1}{W_s}\right| > \varepsilon\right)\\
 =  \PP\left(\left|1-\frac{1}{\hh\, t}\sum_{s=1}^t\log\frac{1}{W_s^\star}\right|>\varepsilon\right)+ o(1),
\end{eqnarray}
where $(W_1^\star,\ldots,W_t^\star)$ are i.i.d.\ with law determined by (\ref{size-biased}). Now, the variable $Y:=\frac{1}\hh \log\frac 1{W_1^\star}$ has mean $1$ by  definition of $\hh$. From \eqref{size-biased} and the assumption (\ref{assume:sparse}), one has $\EE[Y^2]=o(\log n)$. In particular, the variance of $Y$ satisfies
$
{\rm Var}(Y)=o(\log n)
$. Therefore, 
\begin{eqnarray}
\label{cheby}
\EE\left[\left(1-\frac{1}{\hh\,t}\sum_{s=1}^t\log\frac{1}{W_s^\star}\right)^2\right] =\frac1t\,{\rm Var}(Y) \xrightarrow[n\to\infty]{}  0.
\end{eqnarray}
By Chebychev's inequality, \eqref{cheby} and \eqref{weas} already show that $\PP\left((X,W)\in A\right)  \to 0$. However, this is not enough to guarantee the uniform unlikeliness of $A$, due to the extra factor $n$ appearing on the \textsc{rhs} of (\ref{first-moment}). To overcome this difficulty, we will use a more elaborate approach, based on the following higher-order version of (\ref{first-moment}). For any event $B$ in the trajectory space,  for any $\delta>0$ and $k\in\dN$,
\begin{eqnarray}
\label{higher-moment}
\PP\left(\max_{i\in[n]}Q_{i}(B)>\delta\right) \, \leq \, \frac{1}{\delta^k}\,\EE\left[\sum_{i=1}^n\left(Q_{i}(B)\right)^k\right]  \,=\,  \frac{n}{\delta^k}\,\PP\left(\bigcap_{\ell=1}^k\left\{(X^\ell,W^\ell)\in B\right\}\right),
\end{eqnarray}
where the processes $(X^1,W^1),\ldots,(X^k,W^k)$ are formed by generating a random environment $\sigma$ and a uniform state $\cI\in[n]$, and conditionally on that, by running $k$ independent $P_\sigma-$Markov chains in the same environment $\sigma$, with the same starting node $\cI$. We will fix $\delta>0$ and prove that for suitable choices of the event $B$, the right-hand side of (\ref{higher-moment}) is $o(1)$ for \begin{eqnarray}
\label{def:k}
k  :=  \left\lfloor\frac{\log n}{2\log(1/\delta)}\right\rfloor.
\end{eqnarray}  First observe that the variables $(X_s^1,W_s^1)_{0\leq s\leq t},\ldots, (X_s^k,W_s^k)_{0\leq s\leq t}$ can again be constructed sequentially, together with $\sigma$: pick $\cI$ uniformly in $[n]$, set $X_0^1=\cI$, and construct $(X_s^1,W_s^1)_{1\leq s\leq t}$ by repeating $t$ times the instructions $\#1,\#2$ and $\#3$ of subsection \ref{sec:seq}. Then set $X^2_0=\cI$, construct  $(X_s^2,W_s^2)_{1\leq s\leq t}$ similarly (without re-initializing the environment), and so on. Note that $kt$ iterations are performed in total. The union of the graphs induced by the first $\ell$ paths $(X^j_1,\dots,X^j_t)$, $j=1,\dots,\ell$, forms a certain graph $G_\ell=(V_\ell,E_\ell)$, and the argument used for Lemma \ref{lm:tx} shows that  $\tx(G_\ell):=1+|E_\ell|-|V_\ell|$ satisfies
\begin{eqnarray*}
\PP(\tx(G_\ell)\ge 2)  \leq  \frac{(kt)^4}{2n^2} = o\left(\frac{\delta^k}n\right),
\end{eqnarray*}
where we use the fact that by (\ref{def:k}) one has $\delta^k = \Theta(n^{-1/2})$. In view of (\ref{higher-moment}), this reduces our task to showing that $\PP(B_k)=o\left(\frac{\delta^k}n\right)$, where, for any $\ell=1,\dots,k$, we define the event 
\begin{eqnarray}
\label{bk}
B_\ell  :=  \left\{\tx(G_\ell)\le 1\right\} \cap \left\{(X^1,W^1)\in B\right\}\cap\ldots\cap\left\{(X^\ell,W^\ell)\in B\right\}.
\end{eqnarray}
Note that $B_k\subset B_{k-1} \cdots \subset B_1$. 
We will actually show that $\PP\left(B_\ell|B_{\ell-1}\right)=o(1)$ uniformly in $2\leq \ell\leq k$ and that $\PP(B_1)=o(1)$. 
This will be enough to conclude, since 
for $k=\Theta(\log n)$ one has
$$
\PP(B_k)=\PP(B_1)\prod_{\ell=2}^k\PP\left(B_\ell|B_{\ell-1}\right)=o\left(\frac{\delta^k}n\right).
$$ 
To prove Thorem \ref{pr:weight} we now apply the above strategy with two  choices of the event $B$. 
%
%Instead of applying this method directly to the above event $\cA$, we split $\cA$ into two halves.

\bigskip

\noindent {\bf Uniform unlikeliness of $\left\{\rho(t)< e^{-(1+\varepsilon)\,\hh\, t}\right\}$.} 
Define the event 
%By Lemma \ref{lm:low}, we may restrict attention to
\begin{eqnarray*}
B :=  \left\{W_1\cdots W_t < e^{-(1+\varepsilon)\hh\, t}\right\}\cap \left\{\min(W_1,\ldots,W_t)>n^{-\gamma}\right\}, \qquad  \gamma:=\frac{\varepsilon t}{4\ts}. 
\end{eqnarray*}
We use the method described above, i.e., we prove that $\PP(B_1)=o(1)$ and \begin{eqnarray}
\label{bka}\PP\left(B_\ell|B_{\ell-1}\right)=o(1),\end{eqnarray} uniformly in $2\leq \ell\leq k$, with $k$ given by (\ref{def:k}) and $B_k$ defined as in (\ref{bk}). 
Notice that once \eqref{bka} has been proved, the previous observations together with Lemma \ref{lm:low} imply that the event $\left\{\rho(t)< e^{-(1+\varepsilon)\hh\, t}\right\}$ is uniformly unlikely, thus establishing one half of Theorem \ref{pr:weight}.

To prove \eqref{bka}, first observe that $\PP(B_1)$ is bounded from above by \eqref{weas}, so that $\PP(B_1)=o(1)$ follows from \eqref{cheby}. Next, 
fix $2\leq \ell\leq k$, assume that the first $\ell-1$ walks have already been sequentially generated and that $B_{\ell-1}$ holds, and let us evaluate the conditional probability that $(X^\ell,W^\ell)\in A$. 
We distinguish between two scenarios, depending on the random times
\begin{eqnarray*}
\tau := \inf\left\{s\ge 1\colon (X^\ell_{s-1},X^\ell_{s})\notin E_{\ell-1}\right\} & \textrm{ and } &
\tau':=\inf\left\{s\ge 0\colon W^\ell_1\cdots W_s^\ell\le n^{-\gamma} \right\}.
\end{eqnarray*}
Since $n^{-\gamma}= e^{-\varepsilon\,\hh\, t/4}$, we may clearly restrict to the case $t\geq \tau'$, otherwise the event $\rho(t)<e^{-(1+\varepsilon)\hh\, t}$ is trivially false.

\emph{Case I: \;$\tau'<\tau$ and $t\geq \tau'$.} Let $F$ denote the event $\{\tau'<\tau\}\cap\{t\geq \tau'\}$. We show that $\PP(F | B_{\ell-1}) = o(1)$. For any $1\leq s\leq t$, let $\cG_s$ denote the set of directed paths in the graph $G_{\ell-1}$, with  length $s$ and starting node $\cI$.  The condition $\tx(G_{\ell-1})\le 1$ ensures that  $G_{\ell-1}$ is a directed tree with at most one extra edge. Thus, for every vertex $v\in V_{\ell-1}$ there are at most 2 directed paths of length $s$ from the given vertex $\cI$ to $v$. It follows that $|\cG_s|\leq 2|V_{\ell-1}|\leq 2kt$. %check
If $F$ holds, and $\tau'=s$, then $(X_0^\ell,\ldots,X_{s}^\ell)$ is one of the paths in $\cG_s$ 
with weight at most $n^{-\gamma}$. By definition, each such path has conditional probability at most $n^{-\gamma}$ to be actually followed by the $\ell$th walk. Summing over the possible values of $\tau'$, we find that the conditional probability  of $F$ is less than $2kt^2n^{-\gamma}=o(1)$. 

\emph{Case II: \;$\tau\le \tau' \le t$.} Let $F'$ denote the event $\{\tau\le \tau' \le t\}$. We show that $\PP(B_\ell\cap F'  | B_{\ell-1}) = o(1)$. 
On the event $F'$ one has $W_1^\ell\cdots W_{\tau-1}^\ell> n^{-\gamma}$. Since $B$ includes the condition $\min(W_1,\ldots,W_t)>n^{-\gamma}$, and therefore $W_\tau>n^{-\gamma}$,  for $(X^\ell,W^\ell)$ to fall in $B$ we must have
\begin{eqnarray}
\label{cW}
W_{\tau+1}^\ell\cdots W_{t}^\ell & < & n^{2\gamma}e^{-\hh(1+\varepsilon)t} \ = \ e^{-\hh(1+\frac{\varepsilon}{2})t}.
\end{eqnarray}
Now, the condition $j\notin\textrm{Dom}(\sigma_i)$ in line $\#2$ of the sequential generation process is actually satisfied when the $\ell{\textrm{th}}$ walk exits $G_{\ell-1}$, so $X_{\tau}$ is constructed by sampling $\sigma_i(j)$ uniformly in $[n]\setminus\textrm{Ran}(\sigma_i)$.  Since $\sum_i|\textrm{Ran}(\sigma_i)|\le kt$, this random choice and the subsequent ones can be coupled with i.i.d.\ samples from the uniform law on $[n]$ at a total-variation cost less than $\frac{kt^2}{n}=o(1)$. This induces a coupling between $W_{\tau+1}^\ell\cdots W_{t}^\ell$ and a product of (less than $t$) i.i.d.\ variables with law (\ref{size-biased}), and it follows from (\ref{weas})-(\ref{cheby}) that 
 (\ref{cW}) occurs with probability $o(1)$. 

\bigskip

\noindent {\bf Uniform unlikeliness of $\left\{\rho(t)> e^{-(1-\varepsilon)\,\hh\, t}\right\}$.} Let us define the event 
\begin{eqnarray*}
B  :=  \left\{W_1\cdots W_t> e^{-(1-\varepsilon)\hh\, t}\right\}\bigcap \left\{W_1\cdots W_s \le  (\log n)^{-4}\right\}, \qquad  s:=\left\lfloor\frac{\varepsilon t}{2-\varepsilon}\right\rfloor. 
\end{eqnarray*}
We use the same method as above, with this new definition of $B$. Notice that if we prove that $B$ is uniformly unlikely, then it follows from 
Lemma \ref{lm:high} that $\left\{\rho(t)> e^{-(1-\varepsilon)\hh\, t}\right\}$ is also uniformly unlikely, thus completing the  proof Theorem \ref{pr:weight}. 

We need to prove \eqref{bka} with the current definition of the sets $B_j$; see \eqref{bk}. First observe that $\PP(B_1)=o(1)$ follows again as in \eqref{weas}-\eqref{cheby}. Next, 
fix $2\leq \ell\leq k$, assume that the first $\ell-1$ walks have already been sequentially generated and that $B_{\ell-1}$ holds, and let us evaluate the conditional probability that $(X^\ell,W^\ell)\in B$. 
As before, we let $\tau$ be the first exit from $G_{\ell-1}$. We distinguish two cases.

\emph{Case I: $\tau> s$.} We proceed as in case I above. If $B_\ell\cap \{\tau>s\}$ holds, then $(X_0,\ldots,X_s)$ must be one of the paths in the set $\cG_s$, with weight at most $(\log n)^{-4}$. As before, there are less than $2kt$ %check
possible paths, each having conditional probability at most $(\log n)^{-4}$ to be actually followed. Therefore, 
$\PP(B_\ell\cap  \{\tau>s\}|B_{\ell-1}) \leq 2kt(\log n)^{-4}= o(1)$. 

\emph{Case II: $\tau\le s$.} On this event, reasoning as in case II above, one sees that $(W_{s+1}^\ell,\ldots,W_t^\ell)$ can be coupled with $(t-s)$ i.i.d.\ variables with law (\ref{size-biased}) with an error  $o(1)$ in total variation, and  (\ref{weas})-(\ref{cheby}) then implies that their product will be below $e^{-(1-\frac{\varepsilon}{2})\hh(t-s)}$ with probability $1-o(1)$. But $e^{-(1-\frac{\varepsilon}{2})\hh(t-s)}\le e^{-(1-\varepsilon)\hh\, t}$ by our choice of $s$.

\section{Proof of the lower bound in Theorem \ref{th:main}}
\label{sec:lower}

In this section we prove the simpler half of Theorem \ref{th:main}, namely  the lower bound \eqref{cutoff1}. We shall actually prove  \eqref{cutoff1} with $\pi$ replaced by  $\widehat \pi$ given in \eqref{proxy}, as justified in Section \ref{sub:proxy}. 

Fix the environment $\sigma$, an arbitrary probability measure $\nu$ on $[n]$, $t\in\dN$, $\theta\in(0,1)$ and $i,j\in [n]$. Since $P^t(i,j)=Q_i(X_t=j)$, we have
\begin{eqnarray}
\label{ineq}
P^t(i,j)  \geq  Q_i(X_t=j,\rho(t)\le \theta).
\end{eqnarray}
If equality holds in this inequality, then clearly 
\begin{eqnarray*}
\nu(j)-Q_i(X_t=j,\rho(t)\le \theta) \, \leq \, \left[\nu(j)-P^t(i,j)\right]_+,\end{eqnarray*}
where $[x]_+:=\max(x,0)$. On the other-hand, if the inequality (\ref{ineq}) is strict, then 
 there must exist a path of length $t$ from $i$ to $j$ with weight $>\theta$, implying that $P^t(i,j)>\theta$ and hence that
 \begin{eqnarray*}
\nu(j)-Q_i(X_t=j,\rho(t)\le \theta)\, \leq \, \nu(j){\bf 1}_{\{P^t(i,j)>\theta\}}.
\end{eqnarray*}
In either case, we have
\begin{eqnarray*}
\nu(j)-Q_i(X_t=j,\rho(t)\le \theta) \, \leq \, \left[\nu(j)-P^t(i,j)\right]_+ + \nu(j){\bf 1}_{\{P^t(i,j)>\theta\}}.
\end{eqnarray*}
Summing over all $j\in [n]$, the left hand side above yields the probability $Q_i(\rho(t)> \theta)$, while the first term in the right hand side gives the total variation norm $\|\nu-P^t(i,\cdot)\|_{\textsc{tv}}$. 
On the other hand, the Cauchy--Schwarz and Markov inequalities imply
$$
\left(\sum_{j\in[n]}\nu(j){\bf 1}_{\{P^t(i,j)>\theta\}}\right)^2\leq \sum_{j\in[n]}\nu(j)^2 \sum_{\ell\in[n]}{\bf 1}_{\{P^t(i,\ell)>\theta\}} 
\leq \frac1\theta \sum_{j\in[n]}\nu(j)^2 .
$$
Summarizing, \begin{eqnarray}\label{lowerb}
Q_i(\rho(t)> \theta)  \leq  \|\nu-P^t(i,\cdot)\|_{\textsc{tv}} + \sqrt{\frac{1}{\theta}\sum_{j\in [n]}\nu(j)^2}\,.
\end{eqnarray}
We now specialize to $\theta=\frac{\log^3 n}{n}$ and $\nu=\widehat\pi$ as in (\ref{proxy}). If $t = (\lambda + o (1) ) \ts$ with $0< \lambda < 1$ fixed, then for some $\varepsilon>0$ one has $e^{-(1+\varepsilon)\hh\,t} > \theta$ for all $n$ large enough. Therefore, from Theorem \ref{pr:weight}, we have $$\min_{i\in[n]}\,Q_i(\rho(t)> \theta)\,\xrightarrow[n\to\infty]{\bfP} \, 1.$$  
%Here and below, the notation $o_\PP(1)$ stands for any random variable that goes to zero in probability as $n\to\infty$. 
To conclude the proof, it remains to verify that the square-root term in \eqref{lowerb} converges to zero in probability. Below, we prove the stronger
estimate \begin{eqnarray}\label{secm}
\EE\left[\sum_{j\in [n]}\widehat\pi(j)^2\right]  =  o(\theta).
\end{eqnarray}
Fix $h:=\lfloor\frac{\ts}{10}\rfloor$. 
The left-hand-side of \eqref{secm} may be rewritten as $\PP\left(X_{h}=Y_{h}\right)$, where conditionally on the environment $\sigma$, $X$ and $Y$ denote two independent $P_\sigma-$Markov chains, each starting from the uniform distribution on $[n]$. To evaluate this annealed probability, we generate the chains sequentially, together with the environment, as follows: we pick $X_0$ uniformly in $[n]$, and construct $(X_1,\ldots,X_{h})$ by repeating $t$ times the instructions $\#1,\#2$ and $\#3$ of subsection \ref{sec:seq}. We then pick $Y_0$ uniformly at random in $[n]$, and construct $(Y_1,\ldots,Y_{h})$ similarly, without re-initializing the environment. Now, observe that $\{X_{h}=Y_{h}\}\subset \{S\le h\}$, where $$S=\inf\left\{s\ge 0\colon Y_s\in \{X_0,\ldots,X_{h},Y_0,\ldots,Y_{s-1}\}\right\}.$$ By uniformity of the random choices made at each execution of the instruction $\#2$, we have for $0\le s\le h$,
\begin{eqnarray*}
\PP\left(S=s\right)  \leq  \frac{|\{X_0,\ldots,X_{h},Y_0,\ldots,Y_{s-1}\}|}{n} \ \le \ \frac{2h+1}n.
\end{eqnarray*}
By a union bound, we see that $\PP(S\le h) \le \frac{2(h+1)^2}{n}$, which is $o(\theta)$ thanks to our choice of $\theta$. 

\newcommand{\e}{\varepsilon}
\renewcommand{\d}{\delta}
%\newpage

\section{Proof of the upper bound in Theorem \ref{th:main}}
\label{sec:upper}
The overall  strategy of the proof  is similar to that introduced in \cite{bordenave2015random}. 
Before entering  the details of the proof,  let us give a brief overview of the main steps involved. 

Fix the environment and, for every $i,j\in[n]$, define a suitable set of \emph{nice} paths $\cN_t(i,j)$ that go from $i$ to $j$ in $t$ steps, where $t=(\lambda+o(1))\ts$, with $\lambda>1$. Call $P_0^t(i,j)$ the probability that the walk started at $i$ arrives in $j$ after $t$ steps by following one of the paths in $\cN_t(i,j)$. Clearly, $P_0^t(i,j)\leq P^t(i,j)$, and therefore,  for any  probability $\nu$ on $[n]$, any $\d>0$, one 
has 
\begin{align}\label{totvar1}
\|\nu-P^t(i,\cdot)\|_{\textsc{tv}} &=  \sum_{j\in [n]}\left[\nu(j)-P^t(i,j)\right]_+
 \leq  \sum_{j\in [n]}\left[\nu(j)(1+\delta)+\frac{\delta}{n}-P^t_0(i,j)\right]_+ .
%& = &  (1+\varepsilon)\pi_\star(J^c)+1-Q_\star^t(i,J)\\
%  p(i)+2\varepsilon,
\end{align}
Suppose now that, for some $\delta>0$, and some $\nu$,  for all $i,j\in[n]$, one has 
\begin{eqnarray}
\label{main}
P^t_0(i,j)  \leq  (1+\delta)\nu(j)+\frac{\delta}{n}.
\end{eqnarray}
In this case we can compute the sum in \eqref{totvar1} to obtain, for all $i\in[n]$,
 \begin{align}\label{totvar2}
\|\nu-P^t(i,\cdot)\|_{\textsc{tv}} \leq q(i) + 2\delta,
\end{align}
where $q(i)$ is the probability that a walk of length $t$ started at $i$ does not follow one of the nice paths in $\cN_t(i)=\cup_{j}\cN_t(i,j)$, i.e.\ $$q(i)= \sum_{j\in[n]}(P^t(i,j)-P^t_0(i,j)).$$
As explained in Section \ref{sub:proxy}, we want to prove that %PC added math display
\begin{eqnarray}
\label{eq:goalnu}
\max_{i\in[n]}\|\nu-P^t(i,\cdot)\|_{\textsc{tv}}  \xrightarrow[]{\bfP}  0,
\end{eqnarray} 
when $\nu=\widehat\pi$. Thus, roughly speaking, the key to the proof of the upper bound is  to define the set of nice paths $\cN_t(i,j)$ in such a way that: \begin{enumerate}
 \item \label{p11} $q(i)$ vanishes in probability, %uniformly in the starting point $i$, 
 and 
 \item \label{p22} for any $\delta>0$ the bound \eqref{main} holds with high probability if we choose $\nu=\widehat\pi$.
 \end{enumerate} 

The organization of this section is as follows. In Subsection \ref{subsec:fogra}, we will start by defining a forward graph and a forward tree rooted at a vertex. These are then used to define the set of nice paths in Subsection \ref{nice}. In Proposition \ref{notnice}, we will prove that Property \eqref{p11} above holds. In Subsection \ref{subsec:43upbo}, we will prove that Property \eqref{p22} holds (Proposition \ref{upnice}) and conclude the argument.  Throughout this section we will use the following notation;
%PC slight rewording
we refer to Remark \ref{rq:constant} below for more comments on the choice of the constants involved. 

\bigskip
\noindent {\bf Notation}. 
We fix  $0 < \veps <1/20$, and \begin{equation}\label{eq:choicet}
t := (1+ \veps)\, \ts.  %PC removed o(1)
\end{equation} 
Moreover, we set 
\begin{equation}\label{eq:choiceh}
h := \left\lfloor\frac{ \ts}{10}\right\rfloor,\qquad \underline \hh = \hh ( 1 - \tfrac{ \e}{2} )\quad  \hbox{ and }  \quad \overline \hh = \hh ( 1 + \e ).
\end{equation}

\bigskip
For any path ${\rm p}:=(x_0,\dots,x_s)\in[n]^{s+1}$, $s\in\dN$, the weight of ${\rm p}$ is defined by  \begin{equation}\label{weighto}w({\rm p}) = P(x_0,x_1)\cdots P(x_{s-1},x_s).\end{equation}

\subsection{The forward graph $\cG_x(s)$ and the spanning tree $\cT_x(s)$}
\label{subsec:fogra}
%We start with some preliminary steps. 
%Fix the environment $\sigma$. 
For integer $s \geq 1$ and $x\in[n]$ we call $\cG_x(s)$ the weighted directed graph spanned by the set of directed paths ${\rm p}$ with at  most $s$ edges, starting at $x$,   and with weight $w({\rm p})\ge e^{- \overline \hh\, s}$. 
We can construct  $\cG_x(s)$, together with a spanning tree $\cT_x(s)$, as follows.
% formed by paths of minimal lengths from nodes to the root $x$. 
We start at $\cG^0= \cT^0 = x$ and define a process $(\cG^0,\cT^0),  (\cG^1,\cT^1), \dots$, which stops at some random time $\kappa$, and we define $\cG_x(s)= \cG^\kappa$ and $\cT_x(s) = \cT^\kappa$. As in Subsection \ref{sec:seq}, we will add oriented edges one by one, using sequential generation. Initially, $\textrm{Dom}(\sigma_y)=\textrm{Ran}(\sigma_y)=\emptyset$ for all $y\in[n]$. When $j\notin \textrm{Dom}(\sigma_y)$, we interpret $(y,j)$ as a free %PC changed 'half-edge' to 'arrow' and 'accumulative' to 'cumulative' here and below
arrow exiting $y$ to be linked to a node $z$ to be chosen uniformly among the vertices $z\in [n]\setminus\textrm{Ran}(\sigma_y)$.
%Then, $\textrm{Dom}(\sigma_j)$ and $\textrm{Ran}(\sigma_j)$ will grow over time.  
If we are at $(\cG^\ell,\cT^\ell)$, to obtain $(\cG^{\ell+1},\cT^{\ell+1})$ the iterative step is as follows:
\begin{enumerate}[1)]
\item
Consider all nodes $y$ of $\cG^\ell$ together with their free arrows $(y,j)$,  $j \notin \textrm{Dom}(\sigma_y)$. The cumulative weight of such arrows  is  defined as $$\widehat
w(y,j) := w({\rm p})\,p_{y, j},$$
%\prod_{u=0}^{m-1} P(x_u,x_{u+1}) p_{y, j}$ 
where %$(x_0, \ldots , x_m)$ 
${\rm p}$ is the unique path in $\cT^\ell$ from %$x_0 = x$ to $x_m =y$
$x$ to $y$.  Pick $(y,j)$ with maximal cumulative weight $\widehat w(y,j)$, among all free arrows  such that: (i) $y$ is at graph distance at most $s-1$ from $x$, and (ii) the cumulative weight satisfies $\widehat w(y,j)\geq e^{-\overline \hh\, s}$.  If this set is empty, then the process stops and we set $\kappa = \ell$. 
\item Extend $\sigma_y$ by setting  $\sigma_y(j)=z$, where $z$ is uniform in $[n]\setminus \textrm{Ran}(\sigma_y)$.
\item Add the weighted directed edge $(y,z)$, with weight $p_{y,j}$,  to the graph $\cG^{\ell}$; add it also to   $\cT^\ell$ if $z$ was not already a vertex of $\cG^\ell$. This defines  $\cT^{\ell+1}$ and $\cG^{\ell+1}$. 
\end{enumerate}

\smallskip
Notice that $\cT_x(s)$ is a spanning tree of $\cG_x(s)$, and that $\cG_x(s)$ can indeed
be identified with the union of all directed paths ${\rm p}$ with at most $s$ edges, starting at $x$, and such that $w({\rm p})\ge e^{- \overline \hh\, s}$. We start our analysis of $\cG_x(s)$ and $\cT_x(s)$ with a deterministic lemma. 

\begin{lemma}
\label{le:kappa}
Fix $x\in[n]$ and $s\in\dN$, and consider the generation process defined above.  The cumulative  weight $\widehat w_\ell$ of the arrow picked at the $\ell$-th iteration of step $1$ satisfies
$$
\widehat w_\ell\leq \frac{s}{\ell}.
$$
In particular, the random time  $\kappa$ satisfies
$$
\kappa \leq s \,e^{\overline \hh\, s}. 
$$
\end{lemma}
\begin{proof}
Consider the following new tree, say $\tilde \cT ^\ell$, obtained as $\cT^\ell$ in the above process except that at step $3$ if $z$ has already been seen, we create anyway a new fictitious leaf node. Then both $\cG^\ell $ and $\tilde \cT^\ell$ have exactly $\ell$ edges.  Let $F$ denote the set of all leaf nodes
$\tilde \cT^\ell$. Thus $F$ consists of all leaf nodes of $\cT^\ell$ plus all the fictitious leaf nodes introduced above. 
By construction: 
\begin{equation}\label{eq:leaves}
\sum_{{\rm p}:\,x\mapsto F} w({\rm p}) \leq 1, 
\end{equation}
where the sum runs over all directed paths in $\tilde\cT^\ell$ from the root $x$ to a leaf node in $F$. 
 Note also that the chosen cumulative weights at step $1$ for $\ell = 1, 2, \ldots$ are non-increasing: $\widehat w_{\ell-1} \geq \widehat w_\ell$. Hence, any ${\rm p}$ from the sum in \eqref{eq:leaves} satisfies  $w({\rm p}) \geq \widehat w_\ell$. Since there is a unique path ${\rm p}$ for each leaf node in $F$,  
 it follows from \eqref{eq:leaves} that $|F|\widehat w_\ell\leq 1$. Each path ${\rm p}$ has length at most $s$, and their union spans $\tilde \cT^\ell$. Since there are a total of $\ell$ edges one must have $\ell\leq s|F|$. Therefore $\ell \leq s / \widehat w_\ell$ as desired. For the second statement, we use that for $\ell = \kappa$, $\widehat w_\ell\geq e^{-\overline \hh\, s}$. 
\end{proof}

Let as usual $\tx(\cG_x(s)):=1+|E|-|V|$ denote the tree excess of the directed graph $  \cG_x(s)$, where $E$ is the set of edges and $V$ is the set of vertices of $  \cG_x(s)$. Note that $|E|=\kappa$, that $\tx(\cG_x(s))=0$ iff $\cG_x(s)=\cT_x(s)$, and that $\tx(\cG_x(s))\leq 1$ iff $  \cG_x(s)$ is a directed tree except for at most one extra edge. Remark also that if $s\le (1-\veps)\ts$, then the number of vertices in $\cG_x(s)$ satisfies 
$|V|= o(n).$ 
Indeed, there are at most $\kappa+1$ vertices, and by Lemma \ref{le:kappa}, $\kappa\leq \ts e^{\overline \hh(1-\veps) \ts} = \cO(n^{1-\veps^2}\log n)$. 
%by removing some edge of $\cG_x(s)$ one obtains a directed tree. 

\begin{lemma}\label{bplus}
Denote by $S_0$ the set of all $x\in[n]$ such that $\tx( \cG_x(2h))\leq 1$, where $h$ is defined in \eqref{eq:choiceh}. Then with high probability $S_0=[n]$, that is $\PP(S_0=[n])=1-o(1)$. 
\end{lemma}

\begin{proof}
We can use the same argument as in the proof of Lemma \ref{lm:tx}. Consider the  stage $(\cG^\ell,\cT^\ell)\mapsto (\cG^{\ell+1},\cT^{\ell+1})$ of the above sequential generation process. The conditional chance, given the past  stages, that the vertex $z$ in step $3$  is already a vertex of $\cG^{\ell}$ is at most $(\ell +1) / n$. Hence, if $m = \lceil s e^{\overline \hh\, s} \rceil$, from Lemma \ref{le:kappa}, the tree excess of $\cG_x(s)$ is stochastically upper bounded by Binomial$(m, (m+1 )/n)$. As in \eqref{bound:tx}, the probability that the tree excess is larger than $1$ is bounded by 
$$
\frac 1 2  \left(\frac{m(m+1)}{n}\right)^2.
$$
For $s = 2 h$, the later is $o(1/n)$ since $m^4 = o (n)$ which follows from $4 \overline \hh  2 h < (84/100) \log n $ (since $\e < 1/20$). 
\end{proof}

\subsection{Nice trajectories}\label{nice}
We will first show that for most starting states $x \in [n]$, it is likely that the walker spends its first $ (1 - \veps) \ts$ steps in  $\cT_x$ (Lemma \ref{tx1}) and does not come back to it for a long time (Lemma \ref{tx2}). We start by identifying %defining what are 
these good starting points $x$. 

\begin{lemma}[Good states]\label{le:Sstar}
Let  $S_\star$ be the set of all $x \in[n]$ such that $\tx(\cG_x(h))=0$.  For any $s=\Theta(\log n)$, the event $\{ X_s \notin S_\star\}$ is uniformly unlikely.
\end{lemma}

\begin{proof}
In view of (\ref{rk:timeshift}), it is sufficient to prove the claim for $s \leq h$ and $s = \Theta( \log n)$. By Lemma \ref{bplus}, we may further assume that $S_0 = [n]$. Consider the trajectory $(X_0 , \ldots, X_s)$ started at $X_0 = x$. The event that $X_s \notin S_\star$ is contained in the union of the events $A = \{\rho(s) \notin [ e^{-\overline \hh\, s} , e^{-\underline \hh\, s}] \}$ and $A^c  \cap B$ where $ B= \{ (X_0, \ldots X_s ) \in \cP \}$ and $\cP$ is the set of paths starting from $x$  of length $s$ in $\cG_x(2h)$, whose end point is not in $S_\star$.  We claim that $\cP$ has cardinality at most $1$. Assuming this claim, we get  $Q_x ( A^c \cap B) \leq  e^{ - \underline \hh\, s} = o(1)$. Finally, Theorem \ref{pr:weight} asserts that $A$ is uniformly unlikely.

It remains to check the claim. Observe that if $y$ is a vertex of $\cG_x(2h)$ at distance $s \leq h$ from $x$, then any directed edge of $\cG_y (h)$ is also a directed edge of $\cG_x(2h)$. Besides, since $S_0 = [n]$, $\cG _x(2h)$ is a directed tree except for at most one directed edge. If $\cG _x(2h)$ is a directed tree, then obviously $\cP$ is empty from the above observation. %PC changed weak cycle to non-directed cycle and slight rewording in the proof
Now, assume $\cG _x(2h)$ has a no directed cycle, but only one non-directed cycle. Let $y$ be the closest node to $x$ on this cycle ($y$ is unique since there is only one cycle) and let $(x_0 , \ldots, x_u)$, $x_0 = x$, $x_u = y$ be the unique path from $x$ to $y$ in $\cG _x(2h)$. If $s \leq u$, then $\cP  \subset \{ (x_0, \ldots, x_s) \}$, and thus $|\cP|\leq 1$. If $s > u$, then $\cP$ is empty since no forward neighbor $z$ of $y$ at distance $s$ from $x$ can have a cycle in $\cG_z(h)$ (the contrary would create a new cycle since $\cG _x(2h)$ is  no directed cycle). Finally, assume that $\cG _x(2h)$ has a unique directed cycle and let $y$ and $(x_1, \ldots, x_u)$ be as above. If $s \leq u$, then $\cP  \subset \{ (x_0, \ldots, x_s) \}$, and thus $|\cP|\leq 1$. If $s > u$, the only path in $\cP$, if any,  is the path which reaches $y$ (in $u$ steps) and then loop inside the directed cycle during $s-u$ steps.
\end{proof}

%\begin{remark}
%The proof of Lemma \ref{bplus} shows also that $S_\star$ has cardinal number $n - o_{\dP} (n)$. 
%\end{remark}

The next corollary implies that it is enough to check that the upper bound \eqref{eq:goalnu} holds uniformly over  $ S_\star$ rather than over all of $[n]$.  
\begin{corollary}\label{cor:Sstar}
For all integers $u \geq s = \Theta( \log n)$, for any probability $\nu$ on $[n]$: 
$$
\max_{x \in [n]} \| P^u  (x , \cdot ) - \nu \|_{{\textsc{tv}}} \leq  \max_{x \in S_\star} \| P^{u-s}  (x , \cdot ) - \nu \|_{{\textsc{tv}}}   + o_{\bf P}(1),
$$
where $o_{\bf P}(1)$ denotes a random variable that converges to zero in probability, as $n\to\infty$. 
\end{corollary}
\begin{proof}
Notice that 
$$\| P^u  (x , \cdot ) - \nu \|_{{\textsc{tv}}}\leq Q_x(X_s\notin S_\star ) + \max_{y \in S_\star} \| P^{u-s}  (y , \cdot ) - \nu \|_{{\textsc{tv}}}.$$
Taking maximum over $x\in[n]$ and using Lemma \ref{le:Sstar} concludes the proof.
\end{proof}

If ${\rm p}=(x_0,\dots,x_s)$ is a path with $x_0 =x$, we write that ${\rm p} \in \cG_x (s)$ (or ${\rm p} \in \cT_x (s)$) if for any $0 \leq u < s$, $(x_u,x_{u+1})$ is a directed edge of $\cG_x (s)$ (or $ \cT_x (s)$).
Theorem \ref{pr:weight}  implies that the trajectory $(X_0,\ldots, X_s)$ started at $x$  is likely to remain in $\cG_x(s)$ for a long time. We now prove that it is also likely that the trajectory stays  in $\cT_x(s)$  if $x \in S_\star$ and $s$ is not too large. 
\begin{lemma}\label{tx1}
If  $ \e \,  \ts \leq  s \leq (1 - \e) \, \ts $, then 
%$\{ (X_0, \cdots , X_s ) \notin \cT_{x_0} (t) \}$ is uniformly unlikely for $x_0 = S_\star$, that is
\begin{equation}\label{eqtx1}
\max_{x\in S_\star}\,Q_x((X_0,\dots,X_s)\notin\cT_{x}(s))\, \xrightarrow[n\to\infty]{\bfP}  \,0.
\end{equation}
\end{lemma}

\begin{proof}
By construction, there are only two ways that the trajectory exits $\cT_x (s)$: either (i) the weight of the trajectory $\rho(s)$ is below $e^{- \overline \hh\, s}$ or (ii) $(X_0,\cdots , X_s)$ has used an edge in $\cG_x(s) \backslash \cT_x(s)$, that is, there exists $1 \leq u \leq s$ such that $(X_{u-1}, X_u) \in \cG_x(s) \backslash \cT_x(s)$. The event depicted in (i) is uniformly unlikely by Theorem \ref{pr:weight}. We should thus treat the event (ii).

We may follow the argument of \cite[Proposition 12]{bordenave2015random}. Fix $x \in [n]$, and consider the sequential generation process $(\cG^{0}, \cT^{0}), (\cG^{1}, \cT^{1}) , \ldots$  defined above. Define a new process $(M_\ell)_{\ell \geq 0}$ by $M_0 = 0$ and 
$$
M_{\ell+1} = M_\ell + {\bf 1}_{\{\ell < \kappa\}} {\bf 1}_{\{z_\ell \in \cG^ {\ell}\}} \,\widehat w_\ell, 
$$
where $\widehat w_\ell = \widehat w(y_\ell,j_\ell)$ is the cumulative  weight of the arrow $(y_\ell,j_\ell)$ picked in step $1$ and $z_\ell=\sigma_{y_\ell}(j_\ell)$ is the vertex picked in step 2. In words: $M_\ell$ is the sum of cumulative weights of all arrows that are linked to   in $\cG^{\ell} \backslash \cT^{\ell}$. In particular the probability of the scenario described in point (ii) above is bounded above by $M_\kappa$. Thus, to conclude the proof of Lemma \ref{tx1}, it is sufficient to prove that for any fixed $\delta > 0$, $M_\kappa \geq \delta$ is unlikely, uniformly over $x \in S_\star$. By construction, $M_h = 0$ for $x \in S_\star$, hence it is sufficient to prove that $M_{\kappa} - M_h \geq \delta$ is uniformly unlikely.  Note that we may further assume that 
\begin{eqnarray}
\label{extra}
  \widehat{w}_\ell  \le  \frac {\delta}2\,,\qquad \forall \ell\ge h,
\end{eqnarray}
since the complementary event entails the existence of a path of length $h=\Theta(\log n)$ and weight at least $\frac\delta 2=\Omega(1)$ starting at $x$, which is uniformly unlikely by Lemma \ref{lm:high}. In other words, we may safely replace $\widehat{w}_\ell$ with  $\widehat{w}_\ell\wedge \frac \delta 2$ in the definition of $M$, for all $\ell\ge h$. For this modified definition of $M$, this ensures that
\begin{eqnarray}\label{eq:Mellbo}
0 \ \leq \, M_{\ell+1} - M_\ell \, \leq \ \frac{\delta}{2}.
\end{eqnarray}
for all $\ell\ge h$.  We then claim that  for $h$ as in \eqref{eq:choiceh} and any fixed $\delta >0$, uniformly in $x \in [n]$,
\begin{equation}\label{eq:boundM}
\dP ( M_{\kappa} \geq M_{h} + \delta) = o \PAR{ \frac 1 n }.
\end{equation}
%PC added reference to first moment argument
Once we have \eqref{eq:boundM}, the conclusion follows from the first moment argument in \eqref{first-moment}.

To prove \eqref{eq:boundM}, we are going to use a martingale version of Bennett's inequality from \cite{freedman1975tail}. %PC added ref. to Freedman here
Let $\cF_\ell$ be the natural filtration associated to the process $(\cG^{0}, \cT^{0}), (\cG^{1}, \cT^{1}),\dots$.  If $|\cG^\ell| $ is the number of nodes in $\cG^\ell$, then
\begin{eqnarray*}
\dE \SBRA{ M_{\ell + 1} - M_\ell   \Bigm| \cF_\ell  } \,=  \, {\bf 1}_{\{\ell < \kappa\}} \, \frac{\widehat w_\ell  \,|\cG^\ell| }{n - | \textrm{Ran}(\sigma_{y_\ell})|} ,\\
\dE \SBRA{ (M_{\ell + 1} - M_\ell  )^ 2   \Bigm| \cF_\ell  } \, = \, {\bf 1}_{\{\ell < \kappa\}} \, \frac{\widehat w_\ell^2\,  |\cG^\ell| }{n - | \textrm{Ran}(\sigma_{y_\ell})|}.
\end{eqnarray*}
Recall that $|\textrm{Ran}(\sigma_{y_\ell})| \leq \ell$,  $|\cG^\ell| \leq \ell+1$. Moreover,  by Lemma \ref{le:kappa},  $\widehat w_\ell  \leq s / \ell$, and $\kappa \leq s e^{\overline \hh\, s} \leq \ts n^{1 - \e^2}$. %PC added more details 
It follows that $\widehat w_\ell  \,|\cG^\ell|=\cO(\log n)$, $\sum_{\ell \geq h}\widehat w_\ell =\cO((\log n)^2)$. Therefore,
\begin{eqnarray*}
&a:=\sum_{\ell \geq h}\dE  \SBRA{ M_{\ell + 1} - M_\ell   \Bigm| \cF_\ell  }  = \cO\PAR{   (\log n)^2  n^{-\e^2} },  \\
&b:=\sum_{\ell\geq h} \dE \SBRA{ (M_{\ell + 1} - M_\ell  )^ 2   \Bigm| \cF_\ell  }  = \cO \PAR{  (\log n)^3  n^{-1} } .
\end{eqnarray*}
%PC corrected and added more details on application of Freedman's theorem
Next, define 
\begin{equation*}
Z_{\ell+1} = \frac2\d\left(M_{\ell+1} - M_{\ell} - \dE \SBRA{ M_{\ell+1} - M_{\ell}   \Bigm| \cF_\ell  }\right)\,,\qquad \ell\geq h.
\end{equation*}
Thus, $\dE[Z_{\ell+1}|\cF_{\ell}]=0$ and \eqref{eq:Mellbo} implies $|Z_{\ell+1}|\leq 1$, for all $\ell\geq h$. Consider the martingale $\{\phi_u,\,u\geq h\}$ defined by $\phi_h=0$ and  
\begin{equation*}
\phi_u = \sum_{i=h+1}^uZ_i\,,\qquad u> h.
\end{equation*}
Since $M_\kappa  - M_h = a+ \frac\d{2}\phi_{\kappa}$, and $a=o(1)$,  for $n$ sufficiently large one has
 \begin{equation*}
\dP\left( M_\kappa  - M_h \geq 2\delta\right) \leq \dP\left( \phi_{u}\geq 2\, \; \text{for some $u\geq h$}\right).
\end{equation*}
Finally, since the conditional variance of the $Z_i$'s satisfies
$$
b':=\sum_{i\geq h}{\rm Var}\left(Z_{i+1}|\cF_i\right)\leq 4\d^{-2}b\,,
$$
we may use \cite[Theorem 1.6]{freedman1975tail} to estimate 
$$
\dP\left( \phi_{u}\geq 2\, \;\text{for some $u\geq h$}\right)\leq e^2\left(\frac{b'}{2+b'}\right)^{2+b'}\leq (2e\d^{-2}b)^2. 
$$
Since  $b = n^{-1 + o(1)}$,  this concludes the proof of \eqref{eq:boundM}. 
\end{proof}

\begin{lemma}\label{tx2}
Suppose $u,s=\Theta(\log n)$ are such that $ s \leq u \wedge (1-\e) \ts$. Then 
$$\max_{x\in [n]}\,Q_x\left(\{ (X_0,\dots,X_s)\in\cT_{x} (s) \} \cap \{ (X_{s+1},\dots,X_u)\cap\cT_{x}(s)\neq \emptyset \}\right)\, \xrightarrow[n\to\infty]{\bfP}  \,0,$$ 
where $ \{ (X_{s+1},\dots,X_u)\cap\cT_{x}(s)\neq \emptyset \}$ denotes the event that there exists $s+1 \leq v \leq u$ such that $X_v$ is a vertex of $\cT_{x}(s)$.
\end{lemma}

\begin{proof}
We use a version of the method explained in \eqref{higher-moment}, with $k = \cO ( \log n)$ as in \eqref{def:k} and $$B = \{ (X_0,\dots,X_s)\in\cT_{x}(s) \} \cap \{ (X_{s+1},\dots,X_u)\cap\cT_{x}(s)\neq \emptyset \} \cap  \{ \rho(s) \leq e^ { - \underline \hh\, s} \}, $$
Thanks to Theorem \ref{pr:weight}, the intersection with   $\{ \rho(s) \leq e^ { - \underline \hh\, s} \}$ is not restrictive. 
 Consider $k$ %PC slight rewording
 independent trajectories $(X^1,W^1), \ldots , (X^k,W^k)$ with the same initial point $X^\ell_0 = \cI$ for any $1 \leq \ell \leq k$, where $\cI$ is picked uniformly at random in $[n]$. 
 For $1 \leq \ell \leq k$, define the events
$$
B_\ell := \{ (X^1 , W^1) \in B \} \cap \dots \cap \{ (X^\ell , W^\ell) \in B \}.
$$
As explained after  \eqref{bk}, it is sufficient to prove that 
$
\dP( B_\ell | B_{\ell -1} ) = o(1),
$ uniformly in $1 \leq \ell \leq k$. We will show the stronger uniform bounds: $\dP_\cF ( B_1  ) = o(1)$
 and, uniformly  in $2 \leq \ell \leq k$,
\begin{equation}\label{eq:All}
\dP_\cF ( B_\ell | B_{\ell -1}  ) = o(1),
\end{equation}
where $\dP_\cF (\cdot) = \dP ( \cdot | \cF)$ and  $\cF$ is the $\sigma$-algebra generated by the random variables $\cI,\cG_\cI (s)$, and $\cT_\cI (s)$. 
If $B_\ell$ holds, then two disjoint cases may occur: either 
\begin{enumerate}[(i)]
\item
  $(X^\ell_0, \ldots, X^\ell_s)$ equals one of the   trajectories $(X^i_0, \ldots, X^i_s)$, $1 \leq i \leq \ell-1$, in $ \cT_\cI (s)$, or
\item
 $(X^\ell_0, \ldots, X^\ell_s)$ is a new trajectory in $ \cT_\cI (s)$ and $(X^\ell_{s+1}, \ldots , X^\ell_t) \cap \cT_\cI(s)$ is not empty. 
\end{enumerate}
In the case  $\ell=1$ of course only the second scenario occurs.  
If (i) holds, then on the event $B_{\ell-1}$, $(X^\ell_0, \ldots, X^\ell_s)$ is one of the at most $\ell-1$ distinct trajectories in $ \cT_\cI (s)$ each of weight at most  $e^ { - \underline \hh\, s}$. Hence, the probability of this case is upper bounded by $k e^ { - \underline \hh\, s} = o (1)$. If (ii) holds, then the node $X^\ell_s$ has never been visited before and we may couple $(X^{\ell}_{s+1} , \ldots , X^\ell_u) $  with $u-s$ i.i.d.\ samples from the uniform law on $[n]$ at a total-variation cost  less than $\frac{ku^2}{n}=o(1)$; see the  proof of Theorem \ref{pr:weight}. If this coupling occurs, then the chance of intersecting $\cT_\cI (s)$ is at most 
$
(u-s) | \cT_{\cI} (s)  | / n$. The latter is  $o (1)$
since $ | \cT_{\cI} (s)  | \leq s e^{\overline\hh\, s} \leq s n^{1 - \e^2} $ by Lemma \ref{le:kappa}.  This concludes the proof of \eqref{eq:All}. \end{proof}

We turn to  the definition of nice trajectories. Let $\e , h$, and $t$ be fixed as in \eqref{eq:choicet}-\eqref{eq:choiceh}. Set also
$$
s := t - h. 
$$
Since $0 < \e < 1/20$, for $n$ large enough, $$
s \leq (1 - \veps) \ts.
$$
For a given $x \in [n]$ and $y \notin \cG_x(s)$, call $\cG^x_y(h)$ %= \cG_y(h) \backslash \cG_x (s)$ 
the graph spanned by trajectories in $\cG_y(h)$ which do not intersect nodes in $\cG_x( s)$. We denote by $S^x_\star$ the set of $y \notin \cG_x(s)$ such that $\tx(\cG_y^x (h))=0$. The set $\cN_t(x)$ of {\em nice paths} is defined as the subset of all paths ${\rm p}=(x_0, x_1, \ldots, x_t)\in[n]^{t+1}$, 
such that: 
\begin{enumerate}[1)]
\item $w({\rm p})\leq n^{- 1 - \e / 4}$;
\item $x_0= x$ and $(x_0,\dots,x_s)\in\cT_{x}(s)$;
\item $P(x_s,x_{s+1}) \geq n^{-\e /8}$. 
\item $x_{s+1}\in S^x_\star$ and $(x_{s+1},\dots,x_{t} ) \in \cG^x_{x_{s+1}} (h)$.
\end{enumerate}
\smallskip
\smallskip

Combining Lemma \ref{lm:low},  Lemma \ref{le:Sstar}, Lemma \ref{tx1}, Lemma \ref{tx2} and \eqref{rk:timeshift}, we have proved: % the following estimate.
\begin{proposition}\label{notnice}
For $\e, h, s,  t$ as above, we have 
\begin{equation}\label{eqnotnice1}
\max_{x\in S_\star}\,Q_x\left((X_0,\dots,X_t)\notin\cN_t(x)\right)\, \xrightarrow[n\to\infty]{\bfP}  \,0.
\end{equation}
\end{proposition}

\subsection{Upper bound}\label{subsec:43upbo}
Recall that
\begin{equation}\label{upnice1}
P_0^t(x,y)= \sum_{{\rm p}\in\cN_{t}(x,y)}w({\rm p}),
\end{equation}
where $\cN_t(x,y) \subset \cN_t(x)$ is the subset of nice paths such that $x_t = y$.
%We are going to prove the following fact.
\begin{proposition}\label{upnice}
Let $ \e , t$ be as in \eqref{eq:choicet}, and $\widehat \pi$ as in  \eqref{proxy}.
For any $\d >0$, with high probability
\begin{equation}\label{upnice2}
P_0^t(x,y)\leq (1+\d)\widehat\pi(y) + \frac\d{n}\,\qquad \forall x,y\in[n].
\end{equation}
\end{proposition}
Notice that if Proposition \ref{upnice} is available, then the argument in \eqref{main}-\eqref{totvar2} allows us to estimate, with high probability, 
 \begin{align}\label{totvar3}
\|\widehat\pi-P^t(x,\cdot)\|_{\textsc{tv}} \leq q(x) + 2\delta,
\end{align}
where $q(x)=Q_x\left((X_0,\dots,X_t)\notin\cN_t(x)\right)$. From Proposition  \ref{notnice}, uniformly in $x \in S_\star$, one has $q(x) \xrightarrow[]{\bfP}  0$. 
This proves that $\|\widehat\pi-P^t(x,\cdot)\|_{\textsc{tv}}\xrightarrow[]{\bfP}0 $ holds uniformly in $x \in S_\star$. Using Corollary \ref{cor:Sstar}, with e.g.\ $s=\e\ts$ and $u=(1+2\e)\ts$, 
%PC added remark on monotonicity
this implies \eqref{eq:goalnu} with $\nu= \widehat\pi$ and $t=(1+2\e)\ts$, for all $\e\in(0,1/20)$. The latter is sufficient to prove the same estimate for all $t=(\lambda+o(1))\ts$, $\lambda>1$, since the left hand side of \eqref{eq:goalnu} is monotone decreasing in $t\in\dN$ (because of the maximum over $i\in[n]$ this holds for an arbitrary distribution $\nu$). This ends the proof of the upper bound in Theorem \ref{th:main}.

\begin{proof}[Proof of Proposition \ref{upnice}]
%For a path ${\rm p}$ of length $t$, we write ${\rm p}=p_1p_2$ where $p_1$ consists of the first $s=t-h$ steps of $p$ and $p_2$ consists of the last $h$ steps of ${\rm p}$. 
Consider the set ${\mathcal V}_x(s)$ of all  nodes  at distance $s$ from $x$ in the tree $\cT_x(s)$. Any such node must be a leaf by construction. We define the set  $\cL_x(s)$ as the collection of  pairs $(u,k)$, where $u\in{\mathcal V}_x(s)$ and $k\in[n]$. An element of $\cL_x(s)$ is regarded as an arrow $(u,k)$, with cumulative weight $\widehat w(u,k)$. 
%, we write $p_1(x;u,k)$ for the unique path from $x$ to $(u,k)$ contained in the tree $\cT_x(s)$.  
%Given $\cT_x(s)$, a leaf $(u,k)$ in $\cL_x(s)$ is viewed as an half-edge, whose endpoint has not been revealed, and the weight $\widehat w(u,k)$ consists of the product of all weights up to $u$ and including $p_{u,k}$ (as in step $1$ of the construction of $\cT_x (s)$). 
Given $v\in S^x_\star$, by definition there is at most one path of length $h$ from $v$ to $y$ in $\cG_v^x(h)$. If such a path exists, we call it ${\rm p}_\star(v;y)$.   Then, any ${\rm p}\in\cN_t(x,y)$ must be of the form $(x,\dots,u)\circ(u,v)\circ{\rm p}_\star(v;y)$, where $(x,\dots,u)$ is the unique path connecting $x$ to $u$ in $\cT_x(s)$, for some $u\in{\mathcal V}_x(s)$ and  some $v\in S_\star^x$. Here $\circ$ denotes the natural concatenation of paths. Therefore,
\begin{equation}\label{upnice3}
P_0^t(x,y) = \sum_{(u,k)\in\cL_x(s)}\widehat w(u,k)\sum_{v\in S_\star^x} w({\rm p}_\star(v;y)) {\bf 1}_{\{\widehat w(u,k)w({\rm p}_\star(v;y))\leq  n^{- 1 - \e / 4}\}} {\bf 1}_{\{p_{u,k} \geq n^{-\e /8}\}} {\bf 1}_{\{\sigma_u(k)=v\}}.
\end{equation}
Let $\cF$ denote the $\sigma$-algebra generated by all  the random  permutations $\{\sigma_z, z\notin {\mathcal V}_x(s)\}$. A crucial observation is that the quantities $\widehat w(u,k), w({\rm p}_\star(v;y))$, and the sets $\cL_x(s), S_\star^x$ are all $\cF$-measurable. Notice also that by construction one has
\begin{equation}\label{upnice4}
\frac1n\sum_{v\in S_\star^x} w({\rm p}_\star(v;y)) \leq \widehat\pi(y)\,,
\end{equation}
and 
\begin{equation}\label{upnice5}
\sum_{(u,k)\in\cL_x(s)}\widehat w(u,k)\leq 1\,.
\end{equation}
Moreover, conditioned on $\cF$ the remaining permutations $\sigma_u$, $u\in{\mathcal V}_x(s)$, are independent and satisfy $\sigma_u(k)=y$ with probability $1/n$ for all $k,y$. It follows from \eqref{upnice4}-\eqref{upnice5} that 
\begin{equation}\label{eq:boundpi}
\dE_\cF P_0^t(x,y) \leq \widehat\pi(y ),
\end{equation}
where $\dE_\cF$ is the conditional expectation associated to $\cF$.  Notice also that we may write \eqref{upnice3} as
$$
P_0^t(x,y) =\sum_{u\in{\mathcal V}_x(s)} f(u,\sigma_u),
$$
where %the sum runs over all $u \in \cT_x(s)$ at distance $s$ from $x$ and 
$$
f(u, \sigma_u): = \sum_{k = 1}^n  \widehat w(u,k) %\sum_{v\in S_\star^x}
 w({\rm p}_\star(\sigma_u(k);y)) {\bf 1}_{\{\widehat w(u,k)w({\rm p}_\star(\sigma_u(k);y))\leq  n^{- 1 - \e / 4}\}} {\bf 1}_{\{p_{u,k} \geq n^{-\e /8}\}} {\bf 1}_{\{\sigma_u(k)\in S_\star^x\}}. 
$$
Since there are at most $n^{\e / 8}$ indices $k$ such that $p_{u,k} \geq n^{-\e / 8}$, we have %PC minor changes
$$
0 \leq  f(u,\sigma_u)  \leq M =  n^{\e / 8} n^{- 1 - \e / 4} = n^{- 1 - \e / 8}  . 
$$
Thus using Bernstein's inequality (see e.g. \cite[Corollary 2.11]{MR3185193}), for $a>0$
$$
\dP_\cF \PAR{ P_0^t(x,y)  - \dE_\cF P_0^t(x,y) \geq a  } \leq \exp \PAR{ -  \frac{ a^2 }{2  M   \PAR{ \dE_\cF P_0^t(x,y)  + a  } }}.
$$
Applying the above to $a =\d\, \dE_\cF P_0^t(x,y) + \frac\d{n}$ and writing $r=n\dE_\cF P_0^t(x,y)$ one finds
$$
\dP_\cF \PAR{ P_0^t(x,y)  \geq (1+\d)\dE_\cF P_0^t(x,y) + \frac\d{n} } \leq \exp \PAR{ -\frac{ \d ^2  n ^{\e /8}  (r+1)^2}{2 (r(1+\d)+\d  )}}.
$$
Minimizing the exponent over $r\geq 0$ one has that for some constant $c(\d)>0$:
$$
\dP_\cF \PAR{ P_0^t(x,y)  \geq (1+\d)\dE_\cF P_0^t(x,y) + \frac\d{n} } \leq \exp \PAR{ -c(\d) n ^{\e /8} }.
$$
Using \eqref{eq:boundpi} and taking the expectation one obtains
$$
\dP \PAR{ P_0^t(x,y)  \geq (1+\d)\widehat\pi(y) + \frac\d{n} } \leq \exp \PAR{ -c(\d) n ^{\e /8} },
$$
which concludes the proof 
of Proposition \ref{upnice}. 
\end{proof}

\begin{remark}\label{rq:constant}{\em 
%PC slight changes
The range of values for the parameters $\veps$ and $h$ in \eqref{eq:choicet}-\eqref{eq:choiceh}
%choices of the constants $1/20$ in the bound for $\veps$ and  $1/10$ in the definition of $h$ 
is dictated by the need that: a) the forward $h$-neighborhood be typically a tree after $\Theta(\log n)$ steps of the walk (Lemma \ref{le:Sstar}), and b) $t-h$ be smaller than $(1-\veps)\ts$, which guarantees that the walk typically stays on a tree during the first  $t-h$ steps (Lemma \ref{tx1}). One  could have for example replaced $1/10$ by any positive number $(1- \delta)/8$, with $0< \delta < 1/3$ and $1/20$ by $\delta/4$.} \end{remark}

\section{Proof of Theorem \ref{th:heavytails}}
\label{sec:heavytails}
Let $\omega=\left(\omega_{ij}\right)_{1\leq i,j< \infty}$ be i.i.d.\ positive random variables whose tail distribution function $G(t)=\PP(\omega_{ij}>t)$ satisfies (\ref{regvar}) for some $\alpha\in (0,1)$, and consider the random transition matrix 
\begin{eqnarray}
\label{def:Pw}
P_n(i,j) & := & \frac{\omega_{ij}}{\omega_{i1}+\cdots+\omega_{in}}\qquad (1\leq i,j\leq n).
\end{eqnarray}
Permuting entries within a row clearly leaves the distribution  of $P_n$ unchanged. Therefore, $P_n$ is of the form (\ref{def:P}), but with the parameters $(p_{i,j})_{1\le i,j\le n}$ now being random. In order to apply our Theorem \ref{th:main} and obtain Theorem \ref{th:heavytails}, we only have to establish that almost-surely,
\begin{eqnarray}
\label{toshow1}
\frac{1}{n}\sum_{i,j=1}^nP_n(i,j)\log P_n(i,j) \, \xrightarrow[n\to\infty]{} \, h(\alpha);\\
\label{toshow2}
\max_{i\in[n]}\sum_{j=1}^nP_n(i,j)\left(\log P_n(i,j)\right)^2  =  o\left(\log n\right);\\
\label{toshow3}
\limsup_{n\to\infty}\left\{\frac{1}{n}\sum_{i,j=1}^n{\bf 1}_{\{P_n(i,j)>1-\varepsilon\}}\right\} \, \xrightarrow[\varepsilon\to 0^+]{} \,  0.
\end{eqnarray}
The proof will rely on the following estimates on the random probability vector 
$
\left(P_n(1,1),\ldots,P_n(1,n)\right).
$
\begin{lemma}[Uniform sparsity]\label{lm:expUI}For each $\beta\in(\alpha,1)$, there exists $\lambda>0$ such that
\begin{eqnarray}
\label{betaepsilon}
\sup_{n\ge 1}\,\EE\left[\exp\left\{\lambda\sum_{j=1}^n\left(P_n(1,j)\right)^{\beta}\right\}\right]   \, < \,\infty.
\end{eqnarray}
\end{lemma}
\begin{lemma}[Beta asymptotics]\label{lm:beta}Let $\xi_n$ be distributed as a  size-biased pick from the random sequence $\left(P_n(1,1),\ldots,P_n(1,n)\right)$, i.e., for any measurable $g\colon [0,1]\to[0,\infty]$,
\begin{eqnarray*}
\EE\left[g(\xi_{n})\right]  \,=\,  \EE\left[\sum_{j=1}^nP_n(1,j)g\left(P_n(1,j)\right)\right] \, = \, n\EE\left[P_n(1,1)g\left(P_n(1,1)\right)\right]. 
\end{eqnarray*}  
Then $\xi_n\xrightarrow[n\to\infty]{d} \xi$, where 
$\xi$ has the Beta$(1-\alpha,\alpha)-$density: 
\begin{eqnarray*}
f_\alpha(u) =  \frac{(1-u)^{\alpha-1}u^{-\alpha}}{\Gamma(\alpha)\Gamma(1-\alpha)},\qquad (0<u<1).
\end{eqnarray*}
\end{lemma}
Before we establish those Lemmas, let us quickly see how they imply the three almost-sure conditions stated above.  For any $0<\varepsilon,\beta<1$, we have
\begin{eqnarray}
\label{UI}
\sum_{j=1}^nP_n(i,j)\left(\log P_n(i,j)\right)^2 
\, \le  \, (\log \varepsilon)^2+\sup_{p\in[0,\varepsilon]}\left\{p^{1-\beta}(\log p)^2\right\}\sum_{j=1}^n\left(P_n(i,j)\right)^\beta,
\end{eqnarray}
where we have simply split the summands according to whether $P_n(i,j)\le \varepsilon$ or not. Note that the supremum on the right-hand side can be made arbitrarily small by choosing $\varepsilon$ small enough. Claim (\ref{toshow2}) follows, since for $\beta>\alpha$, Lemma \ref{lm:expUI} ensures that almost-surely as $n\to\infty$,
\begin{eqnarray}
\max_{i\in[n]}\left\{\sum_{j=1}^n\left(P_n(i,j)\right)^\beta\right\}  \, = \, \cO(\log n).
\end{eqnarray}

We now turn to (\ref{toshow1}). The row entropies $\left\{-\sum_{j=1}^nP_n(i,j)\log P_n(i,j)\right\}_{1\le i\le n}$ are independent, $[0,\log n]-$valued random variables with mean $-\EE[\log\xi_n]$, where $\xi_n$ is as in Lemma \ref{lm:beta}. Therefore, Azuma-Hoeffding's inequality ensures that almost-surely as $n\to\infty$,
\begin{eqnarray*}
\frac{1}{n}\sum_{i,j=1}^nP_n(i,j)\log P_n(i,j)   & = & \EE[\log \xi_n] + o(1).
\end{eqnarray*}
In view of (\ref{UI}), Lemma \ref{lm:expUI} is more than enough to ensure the uniform integrability of  $(\log \xi_n)_{n\ge 1}$. Together with the weak convergence $\xi_n\to\xi$ stated in Lemma \ref{lm:beta}, this implies
\begin{eqnarray*}
\EE\left[\log \xi_n\right] & \xrightarrow[n\to\infty]{} & \EE\left[\log \xi\right].
\end{eqnarray*}
It is classical that the expected logarithm of a Beta$(1-\alpha,\alpha)$ is $\psi(\alpha)-\psi(1)=-h(\alpha)$, and (\ref{toshow1}) follows.

 The proof of (\ref{toshow3}) is similar: for each $\varepsilon<\frac 12$, the random variables $\left\{\sum_{j=1}^n{\bf 1}_{\{P_n(i,j)\ge 1-\varepsilon\}}\right\}_{1\le i\le n}$ are independent, $[0,1]-$valued and with mean $\EE[\xi_n^{-1}{\bf 1}_{\{\xi_n\ge1-\varepsilon\}}]$. Therefore, Azuma-Hoeffding's inequality ensures that almost-surely as $n\to\infty$, 
\begin{eqnarray*}
\frac{1}{n}\sum_{i,j=1}^n{\bf 1}_{\{P_n(i,j)\ge 1-\varepsilon\}}   & = & \EE[\xi_n^{-1}{\bf 1}_{\{\xi_n\ge1-\varepsilon\}}] + o(1)\\
& = & \EE[\xi^{-1}{\bf 1}_{\{\xi\ge1-\varepsilon\}}] + o(1),
\end{eqnarray*}
where the second line follows from Lemma \ref{lm:beta} and the fact that the Beta distribution is atom-free. It remains to note that $\EE[\xi^{-1}{\bf 1}_{\{\xi\ge1-\varepsilon\}}]\to 0$ as $\varepsilon\to 0$, since $\PP\left(\xi\in (0,1)\right)=1$. We now turn to the proof of Lemmas \ref{lm:expUI} and \ref{lm:beta}.
\subsection{Proof of Lemma \ref{lm:expUI}}

Our starting point is a classical result on regularly varying functions (see, e.g., \cite[Theorem VIII.9.1]{MR0270403}), which asserts that as $t\to\infty$,
\begin{eqnarray*}
\EE\left[\left({\omega_{11}}\wedge t\right)^{\beta}\right]  & \sim & \frac{\beta}{\beta-\alpha}t^{\beta}\,\PP\left(\omega_{11}>t\right).
\end{eqnarray*}
In particular, there exists a constant $c_\beta<\infty$ such that for all $t>0$, 
\begin{eqnarray*}
\EE\left[\left(\frac{\omega_{11}}{t}\wedge 1\right)^{\beta}\right]  \leq c_\beta \,\PP\left(\omega_{11}>t\right). 
\end{eqnarray*}
Since $(\omega_{11},\ldots,\omega_{1n})$ are i.i.d., we immediately obtain that for any $J\subset[n]$,
\begin{eqnarray*}
\EE\left[\prod_{j\in J}\left(\frac{\omega_{1j}}{t}\wedge 1\right)^{\beta}\right]  \leq c_\beta^{|J|}\, \PP\left(\min_{j\in J}\omega_{1j}>t\right). 
\end{eqnarray*}
This formula holds for any $t>0$, and we may choose $t=\max_{j\in[n]\setminus J}\omega_{1j}$, since the latter is independent of $(\omega_{1j})_{j\in J}$. With this choice of $t$, we clearly have $P_{n}(1,j)\leq \frac{\omega_{1j}}t\wedge 1$ and therefore
\begin{eqnarray*}
\EE\left[\prod_{j\in J}\left(P_n(1,j)\right)^{\beta}\right] \leq  c_\beta^{|J|}\, \PP\left(\min_{j \in J}\omega_{1j}>\max_{j\in[n]\setminus J}\omega_{1j}\right).
\end{eqnarray*}
Write $A_J$ for the event on the right-hand side. Clearly, the $\{A_J\colon J\subset[n],|J|=k\}$ are pairwise disjoint. Thus,
\begin{eqnarray}
\label{cbeta}
\sum_{J\subset[n],|J|=k}\EE\left[\prod_{j\in J}\left(P_n(1,j)\right)^{\beta}\right] \leq  c_\beta^{k}.
\end{eqnarray}
This is enough to conclude. Indeed, using $e^{\lambda x}\leq 1+(e^\lambda-1) x$ for  $x\in[0,1]$, we have
\begin{align*}
\EE\left[\exp\left\{\lambda\sum_{j=1}^n\left(P_n(1,j)\right)^{\beta}\right\}\right] & \leq  \EE\left[\prod_{j=1}^n\left(1+(e^\lambda-1)\left(P_n(1,j)\right)^{\beta}\right)\right]\\  
& \leq  \sum_{k=0}^n(e^\lambda-1)^k\sum_{J\subset[n],|J|=k}\EE\left[\prod_{j\in J}\left(P_n(1,j)\right)^{\beta}\right]\\
 & \leq   \sum_{k=0}^n\left(c_\beta(e^\lambda-1)\right)^k,
\end{align*}
which is bounded uniformly in $n$ as long as $\lambda<\log\left(1+\frac 1{c_\beta}\right)$. 

\subsection{Proof of Lemma \ref{lm:beta}}
Since the $(\xi_n)_{n\ge 1}$ are $[0,1]-$valued, it is enough to prove 
$\EE\left[\xi_{n}^p\right]\to \EE\left[\xi^p\right]$ for each $p\ge 0$. We first rewrite both sides as follows:
\begin{eqnarray}
\label{betaleft}
\EE\left[\xi_{n}^p\right] & = & \int_0^1{n}\PP\left(\left\{P_n(1,1)\right\}^{p+1}>u\right)du \ = \ \int_0^1{n}\PP\left(P_n(1,1)>u\right)(p+1)u^{p}du, \\
\label{betaright}
\EE[\xi^p] & = & \int_0^1  \frac{f_\alpha(u)}{u}u^{p+1}  du\ = \  \int_0^1  \kappa\left(\frac{1-u}{u}\right)^\alpha(p+1)u^p  du,
\end{eqnarray}
where $\kappa^{-1}=\Gamma(1+\alpha)\Gamma(1-\alpha)$, and where we have used the change of variables $u\mapsto u^{p+1}$ for (\ref{betaleft}) and an integration by parts for (\ref{betaright}). Comparing these two lines,  our goal reduces to proving that
\begin{eqnarray}
\label{claim:u}
\forall u\in(0,1),\qquad n\PP\left(P_n(1,1)>u\right) & \xrightarrow[n\to\infty]{} & \kappa\left(\frac{1-u}{u}\right)^{\alpha}.
\end{eqnarray}
Indeed, the convergence of (\ref{betaleft}) to (\ref{betaright}) then follows by dominated convergence since for $\beta\in(\alpha,1)$,  
\begin{eqnarray}
\label{betadom}
n\PP\left(P_n(1,1)>u\right)  \,= \, \EE\left[\sum_{j=1}^n{\bf 1}_{\{P_n(1,j)>u\}}\right]\, \leq \,  u^{-\beta}\EE\left[\sum_{j=1}^n\{P_n(1,j)\}^\beta\right] \, \leq \, c_\beta u^{-\beta},
\end{eqnarray}
by (\ref{cbeta}). We may now fix $0<u<1$ and focus on (\ref{claim:u}). Our regular variation assumption on $G$ yields
\begin{eqnarray}
\label{varphis}
R(s) \, := \, \frac{G\left(\frac{us}{1-u}\right)}{G(s)} & \xrightarrow[s\to\infty]{} & \left(\frac {1-u}u\right)^{\alpha}.
\end{eqnarray}
In particular, $s\mapsto R(s)$ is bounded on $(0,\infty)$. Now, since $\omega_{11}$ is independent of $S_n:=\omega_{12}+\cdots+\omega_{1n}$,
\begin{eqnarray*}
\PP\left(P_n(1,1)>u\right) \, = \, \PP\left(\omega_{11}>\frac{u S_{n}}{1-u}\right)
\, = \, \EE\left[G\left(\frac{u S_{n}}{1-u}\right)\right]\, = \, \EE\left[G(S_{n})R (S_{n})\right].
\end{eqnarray*}
Observing that the right-hand side simplifies to $\EE\left[G(S_{n})\right]$ when $u=\frac 12$, we deduce that
\begin{eqnarray*}
\PP\left(P_n(1,1)>u\right)-\left(\frac {1-u}u\right)^{\alpha}\PP\left(P_n(1,1)>\frac{1}{2}\right) \, = \,  \EE\left[G\left(S_{n}\right)\left\{R(S_{n})-\left(\frac {1-u}u\right)^{\alpha}\right\}\right].
\end{eqnarray*}
Since $S_{n}$ increases almost-surely to $+\infty$ as $n\to\infty$ and since $R$ is bounded, (\ref{varphis}) implies that
\begin{eqnarray*}
\EE\left[\left\{R(S_{n})-\left(\frac {1-u}u\right)^{\alpha}\right\}^2\right] & \xrightarrow[n\to\infty]{} & 0,
\end{eqnarray*}
by dominated convergence. On the other hand, since $G$ is decreasing, we have 
\begin{eqnarray*}
\EE\left[\left\{G\left(S_{n}\right)\right\}^2\right] \, \leq
 \, \EE\left[\left\{G\left(\max_{2\leq j \leq n}\omega_{1j}\right)\right\}^2\right] \, = \, \PP\left[\min(\omega_{11},\omega_{1(n+1)})>\max_{2\leq j \leq n}\omega_{1j}\right] \, 
 \leq \, {n+1 \choose 2}^{-1},
 \end{eqnarray*}
 by symmetry. Invoking the Cauchy-Schwarz inequality, we conclude that
\begin{eqnarray*}
n\PP\left(P_n(1,1)>u\right) - \left(\frac {1-u}u\right)^{\alpha}n\PP\left(P_n(1,1)>\frac{1}{2}\right) & \xrightarrow[n\to\infty]{} & 0.
\end{eqnarray*}
This is not quite (\ref{claim:u}), as it is not yet clear that $n\PP\left(P_n(1,1)>\frac{1}{2}\right)\to\kappa$. However, one may still insert this into (\ref{betaleft}) and invoke the domination (\ref{betadom}) to obtain that
\begin{eqnarray*}
\EE\left[\xi_{n}^p\right] - \frac{n}{\kappa}{\PP\left(P_n(1,1)>\frac{1}{2}\right)}\EE[\xi^p]  & \xrightarrow[n\to\infty]{} & 0.
\end{eqnarray*}
But now the  special case $p=0$ shows that $n\PP\left(P_n(1,1)>\frac{1}{2}\right)\to\kappa$, which completes the proof.  

%PC updated biblio

 \bibliographystyle{abbrv}
\bibliography{sparse}

\end{document}